\documentclass[12pt]{amsart}
\usepackage{amssymb,amsmath,amsthm,amsfonts,amscd,euscript,hyperref,fourier,bm,tikz-cd}
\usepackage[in]{fullpage}
\usepackage[all]{xy}

\theoremstyle{plain}

\newtheorem{lemma}{Lemma}[section]
\newtheorem{theorem}[lemma]{Theorem}
\newtheorem{corollary}[lemma]{Corollary}
\newtheorem{proposition}[lemma]{Proposition}

\theoremstyle{definition}

\newtheorem{definition}[lemma]{Definition}
\newtheorem{remark}[lemma]{\sc Remark}

\newcommand{\ac}{\scriptstyle \text{\rm !`}}

\DeclareMathAlphabet{\pazocal}{OMS}{zplm}{m}{n}

\def\calC{\pazocal{C}}

\def\calE{\pazocal{E}}
\def\calF{\pazocal{F}}
\def\calG{\pazocal{G}}

\def\calK{\pazocal{K}}
\def\calL{\pazocal{L}}
\def\calM{\pazocal{M}}

\def\calO{\pazocal{O}}

\def\calT{\pazocal{T}}

\def\calV{\pazocal{V}}

\DeclareMathOperator{\Lie}{Lie}
\DeclareMathOperator{\Com}{Com}
\DeclareMathOperator{\HyperCom}{HyperCom}
\DeclareMathOperator{\Grav}{Grav}
\DeclareMathOperator{\Gerst}{Gerst}

\DeclareMathOperator{\Conf}{Conf}

\DeclareMathOperator{\Hom}{Hom}

\DeclareMathOperator{\Der}{Der}
\DeclareMathOperator{\colim}{colim}

\DeclareMathOperator{\CGK}{\calC\calG\calK}
\DeclareMathOperator{\bCGK}{\overline{\calC\calG\calK}}
\DeclareMathOperator{\hCGK}{\widehat{\calC\calG\calK}}

\DeclareMathOperator{\D}{\mathsf{D}}
\DeclareMathOperator{\fD}{\mathsf{fD}}

\DeclareMathAlphabet{\mathbbold}{U}{bbold}{m}{n}

\def\k{\mathbbold{k}}

\usepackage[backend=bibtex,
    sorting=anyt,
    isbn=false,
    url=false,
    doi=false,
    maxbibnames=100
    ]{biblatex}
\begin{document}

\title{Operads of moduli spaces of points in $\mathbb{C}^d$ revisited}
\author{Vladimir Dotsenko}
\address{Institut de Recherche Math\'ematique Avanc\'ee, UMR 7501, Universit\'e de Strasbourg et CNRS, 7 rue Ren\'e-Descartes, 67000 Strasbourg, France}
\email{vdotsenko@unistra.fr}

\author{Eduardo Hoefel}
\address{Departamento de Matemática, Universidade Federal do Paraná, C.P. 019081, 81531-990 Curitiba, Paraná, Brazil}
\email{hoefel@ufpr.br}

\author{Sergey Shadrin}
\address{Korteweg-de Vries Institute for Mathematics, University of Amsterdam, P. O. Box 94248, 1090 GE Amsterdam, The Netherlands}
\email{s.shadrin@uva.nl}

\author{Grigory Solomadin}
\address{Department of Mathematics and Computer Science, Philipps-Universität Marburg, Hans-Meerwein-Strasse 6,
D-35043 Marburg}
\email{grigory.solomadin@gmail.com}

\date{}

\begin{abstract} 
We study the operad structure on the homology of moduli spaces of pointed rooted trees of $d$-dimensional projective spaces, introduced by Chen, Gibney and Krashen a couple of decades ago. We describe this operad by generators and relations, show that it is homotopy Koszul, exhibit a Givental-type action on representations of that operad, and prove that this operad represents the homotopy quotient of the operad of chains of $S^1$-framed little $2d$-disks by its natural circle action. Our approach also sheds new light on the $d=1$ case, revealing a new combinatorial way to write the original Givental formulas.  
\end{abstract}

\dedicatory{To the memory of Yuri Ivanovich Manin}
\maketitle

\tableofcontents

\section*{Introduction}

In late 1960s, a remarkable mathematical structure emerged independently in many different corners of the world: various versions of it were considered by Boardman and Vogt \cite{MR236922} under the name ``categories of operators in standard form'', by Artamonov \cite{MR0237408} in universal algebra under the name ``clones of multilinear operations'', by Kelly \cite{MR340372,MR340371} under the name ``club'', and, last but not the least, by May \cite{MR0420610} who gave it a beautiful name ``operad''. An operad abstracts the properties exhibited by the collection $\{\mathrm{Map}(X^n,X)\}$ under substitutions of multi-variable operations into each other and permutations of arguments. 

\subsection*{Some celebrated topological operads}
Among operads of topological origin, there are two of particular notoriety: the operad of little disks $\D_2$, known already to Boardman and Vogt \cite{MR236922} who however used cubes instead of disks, and the Deligne--Mumford operad $\overline{\calM}$ made of compactifications of the moduli spaces of genus zero curves with marked points, where the operad structure seems to have first appeared in print in the work of Beilinson and Ginzburg \cite{MR1159447}. These operads turn out to be related to each other and to the compactifications of configuration spaces of Axelrod--Singer \cite{MR1258919} and of Fulton--MacPherson \cite{MR1259368}; one may say that the work of many authors who unravelled various aspects of that relationship (such as Beilinson and Ginzburg \cite{MR1159447}, Getzler and Jones \cite{getzler1994operadshomotopyalgebraiterated}, Ginzburg and Kapranov \cite{MR1301191}, Kontsevich \cite{MR1718044}, and Markl \cite{MR1684178} to name but a few), was an important driving force behind what Loday \cite{MR1423619} called the ``renaissance of operads'' of the 1990s. The main aspect of the relationship between these operads, as understood nowadays, can be summarized by the diagram
 \[
 \xymatrix@M=6pt{
    \overline{\calM}\ar@{<--}^{/_h S^1}[rr]  & &  \fD_2  \\ 
    \calM\ar@{^{(}->}[u]\ar@{<<-}^{/ S^1}[rr]  & & \D_2 \ar@{^{(}->}[u]     
 }
 \]
The two objects in this diagram whose definitions we have not yet recalled are the operad of framed little disks $\fD_2$, which is the extension $\D_2\rtimes S^1$ of the operad $\D_2$ by the natural circle action on it, and the open moduli spaces of genus zero curves with marked points $\calM=\big\{\calM_{0,n+1}\big\}$. The relationships between these objects are as follows. Each open moduli space $\calM_{0,n+1}$ is homotopic to the quotient of the component $\D_2(n)$ by the circle action. That open moduli space is contained in its compactification $\overline{\calM}_{0,n+1}$, and the operad structure on $\overline{\calM}$ comes from particular geometric features of the boundary divisor. Finally, the operad $\overline{\calM}$ turns out to represent the homotopy quotient of $\fD_2$ by the natural action of $S^1$.

\subsection*{New results}
In this paper, we exhibit a generalization of all the above phenomena for arbitrary even dimension. For that, we have to replace the open moduli spaces $\calM_{0,n+1}$ and their Deligne--Mumford compactifications $\overline{\calM}_{0,n+1}$ by the spaces $\CGK_d(n)$ that parametrize configurations of $n$ points in $\mathbb{C}^d$ modulo translations and homotheties, as well as their compactifications $\bCGK_d(n)$ that parametrize moduli of stable rooted trees of $d$-dimensional projective spaces, where each projective space $\mathbb{C}P^d$ has a distinguished hyperplane $H$ and a few marked points in $\mathbb{C}P^d\setminus H$; these geometric objects were first studied by Chen, Gibney and Krashen \cite{MR2496455}, which we reflect in our choice of notation. In the paper \cite{MR3068956}, which as the title of the present paper suggests, was one of important inspirations for some parts of our work, Westerland gave a description of the operad structure on $\{H_{\bullet-1}(\CGK_d)\}$, analogous to the gravity operad of Getzler \cite{MR1284793,MR1363058}, and remarked ``We expect this computation to be useful in determining the structure of the homology of the operad $\{H_{\bullet}(\overline\CGK_d)\}$''. We bring this expectation to life. The diagram 
 \[
 \xymatrix@M=6pt{
    \bCGK_{d}\ar@{<<--}^{/_h S^1}[rr]  & &  \D_{2d}\rtimes S^1   \\ 
    \CGK_d\ar@{^{(}->}[u]\ar@{<<-}^{/ S^1}[rr]  & & \D_{2d} \ar@{^{(}->}[u]     
 }
 \]
looks deceptively similar, yet it comes with a lot of surprises. First, the ``usual'' framed little disk operad, first studied in detail by Salvatore and Wahl \cite{MR1989873}, is $\fD_m:=\D_m\rtimes SO(m)$, where the full rotational symmetry of the disk is used, so it is not at all apparent why considering only the diagonal embedding $S^1=SO(2)\hookrightarrow SO(2d)$ is the right thing to do. Moreover, the left column of this diagram behaves much more subtly on an algebraic level for $d>1$, primarily due to the fact that the mixed Hodge structure of $\CGK_d(n)$ is not pure for such values of $d$; as we shall see, this is closely related to the fact that the operad $\D_{2d}\rtimes S^1$ is not formal, in contrast with the operad $\D_{2d}\rtimes SO(2d)$ that is shown to be formal by Khoroshkin and Willwacher \cite{khoroshkin2025realmodelsframedlittle}. The homotopy quotient result hints that there is a rich group of symmetries \emph{à la} Givental \cite{MR1901075,MR2115767} acting on representations of the operad 
 \[
\HyperCom_d:=H_\bullet(\bCGK_{d}),    
 \]
and indeed we were able to define such an action. Our approach has two noteworthy features. First of all, we were able to give a definition of Givental symmetries in higher generality, producing an infinitesimal action of the Lie algebra $\Der(\calO)[[z]]$ (as opposed to its Lie subalgebra $\calO(1)[[z]]$ for which the action is usually defined) on 
 \[
\Hom_{\mathrm{Op}}(\HyperCom_d,\calO).    
 \]
Second, our approach allows for new simple formulas even in the well understood $d=1$ case: in a sense, a lot of computations with the Givental formulas use nothing but the combinatorics of Boolean lattices. Similarly to how the Givental action in the classical case was shown by two of the authors and Vallette~\cite{MR3019721} to encode the Grothendieck--Koszul definition of the differential order of an operator on a commutative associative algebra, in our case we unravel a new notion of a ``homotopy circle action of differential order $\frac{d+1}{d}$''. It is worth indicating connections between our work and two other recent papers: a relationship between $\bCGK_{d}$ and $\D_{2d}\rtimes S^1$ in the context of logarithmic geometry was exhibited by Lindström~\cite{lindström2025loggeometricmodelslittle}, while the work of Nesterov \cite{nesterov2025hilbertschemespointsfultonmacpherson} defines $\psi$-classes on Fulton--MacPherson spaces, which, once restricted to a suitable stratum, can be shown to lead to an alternative definition of the Givental action.    
 
It turns out that our results admit an interesting limiting behaviour as $d\to\infty$, which in fact addresses one of the original motivations for our work. In \cite{MR4580527}, two of the authors of the present paper and Tamaroff considered the following question. Suppose that $A$ is a commutative differential graded algebra equipped with an action of the algebra $H_\bullet(S^1)$ such that the fundamental class of $S^1$ acts as a differential operator of order $k\in\mathbb{N}$, and suppose that this action is homotopically trivialized. What kind of structure does this induce on the homology of $A$? For $k=2$, this structure is precisely a structure of an algebra over the homology of the Deligne--Mumford operad mentioned above, but in fact already for the simpler case $k=1$ one obtains a very rich structure: there are infinitely many binary operations of all possible even degrees satisfying together certain ``associativity equations''. Analyzing those equations, we felt that there should be a topological operad inducing those equations, with the components being certain tree-shaped wedges of copies of $\mathbb{C}P^\infty$, particular cases of $\Gamma$-wedges of Anick~\cite{MR764587} and of the spaces studied by Davis and Januszkiewicz~\cite{MR1104531}. As a byproduct of our work, we were able to confirm this expectation.  

\subsection*{Further questions}

In~\cite{MR4580527} two authors of the present paper and Tamaroff emphasized that the classical $d=1$ case is a part of an ``Arnold's trinity of algebraic gravity packages'': if one replaces the symmetric operad of Deligne--Mumford spaces by the nonsymmetric operad of toric varieties of associahedra~\cite{MR4072173} or by twisted associative algebra of Losev--Manin moduli spaces~\cite{MR1786500}, there is a suitable analogue of the operad of little disks and suitable diagrams of algebraic and geometric objects and their interrelationships. It seems likely that the results of the present paper can be generalized to extend this trinity for arbitrary $d$ (in fact, the corresponding higher dimensional Losev--Manin spaces appear in recent work of Gallardo, González-Anaya, González, and Routis \cite{gallardo2023higherdimensionallosevmaninspacesgeometry}); however, checking all the necessary details would make the present paper too long, so we chose to not discuss this extension here. 

We also wish to raise the following question. Examining the proof of Theorem \ref{th:CGK}, one sees that the beginning of our proof works \emph{mutatis mutandis} in the full generality of the Feichtner--Yuzvinsky operad of Coron \cite[Sec.~4.1.3]{MR4948093}. It is natural to ask whether that theorem holds in this level of generality; we leave exploring this direction to an interested reader. 

Finally, an important question that we have not addressed yet is to exhibit interesting algebras over the operad $H_\bullet(\bCGK_{d})$. For $d=1$, the most appealing class of examples comes from genus zero Gromov--Witten invariants, see, e.g., the beautiful paper of Manin \cite{MR1701927}. It is not entirely clear what are the most meaningful examples for $d>1$. One hint of promising direction to explore is proposed by Manin and Marcolli \cite[Sec.~2.6]{MR3261871}; they argue that for $d=3$, the space of trees of pointed three-dimensional projective spaces as a certain ``multiverse landscape'', where each individual $\mathbb{C}P^3$s is viewed as a twistor space of a $4$-dimensional complex spacetime. Note that the Chen--Gibney--Krashen spaces use both points and hyperplanes in $\mathbb{C}P^d$, and for $d=3$ we may use the Penrose transform (see Manin's monograph \cite[Sec.~1.3]{MR1632008} for details), under which points and hyperplanes  to ``$\alpha$-planes'' and ``$\beta$-planes'' of the Klein--Plücker quadric in $\mathbb{C}P^5=\mathbb{P}(\Lambda^2(\mathbb{C}^4))$, so that a choice of a hyperplane $H$ and several marked points in $\mathbb{C}P^3\setminus H$ corresponds to choosing several $\alpha$-planes that do not meet a chosen $\beta$-plane. Gluing together copies of $\mathbb{C}P^3$ into a rooted tree,  while not a twistor space itself, may be deformed to a twistor space \emph{à la} Donaldson--Friedman~\cite{MR994091}. Unravelling the kind of invariants arising from this viewpoint is a very interesting challenge that will be addressed elsewhere.

\subsection*{Structure of the paper}

This paper is organized as follows. In Section~\ref{sec:rec}, we give necessary recollections on operads and De Concini--Procesi wonderful compactifications. In Section~\ref{sec:DM-CGK}, we discuss the Chen--Gibney--Krashen spaces from the point of view of wonderful compactifications, in particular spelling out the induced operad structures on their (co)homology. In Section~\ref{sec:gravity}, we study in detail the $d$-dimensional version of the gravity operad; in particular, we give a presentation of that operad by generators and relations (Theorem \ref{th:GravIso}), and prove that it is Koszul. In Section~\ref{sec:CGK}, we give a presentation of the operad 
$\HyperCom_d$ by generators and relations (Theorem \ref{th:CGK}), and prove that it is a homotopy Koszul operad whose Koszul dual is obtained from the gravity operad by imposing a nontrivial higher structure (Theorem~\ref{th:MHShigher}). In Section~\ref{sec:S1framed}, we discuss the $S^1$-framed version of the operad $\D_{2d}$, and exhibit a small model of that operad (Proposition\ref{prop:d-model}), which plays the role of the operad of Batalin--Vilkovisky algebras for $d>1$. In Section~\ref{sec:Givental}, we introduce a Givental-type action on representations of the operad $\HyperCom_d$ (Proposition \ref{prop:GiventalDer}), explain how it is related to a ``more standard'' action defined using suitable $\psi$-classes, and use it to define ``differential operators of order $\frac{d+1}{d}$'' (Proposition~\ref{prop:GiventalDO}). In Section~\ref{sec:CGKHomotopyQuotient}, we prove that the Chen--Gibney--Krashen operad represents the homotopy quotient of the operad $\D_{2d}\rtimes S^1$ by its natural circle action (Theorem \ref{th:homotopy-q}). Finally, in Section~\ref{sec:DJcolimit} we discuss what happens as $d\to\infty$ and, in particular, express the Davis--Januszkiewicz spaces as colimits of Chen--Gibney--Krashen spaces, up to a homotopy (Theorem \ref{thm:spheq}).

\subsection*{Conventions}

We use the ``topologist's notation'' $\underline{n}$ for $\{1,2,\ldots,n\}$. For $I\subseteq\underline{n}$, we denote $I^c:=\underline{n}\setminus I$. All vector spaces are defined over the ground field $\k$ of zero characteristic. Vector spaces we consider usually carry homological degrees, and we view them as chain complexes with zero differential; an important consequence of the presence of homological degrees is the Koszul sign rule for symmetry isomorphisms of tensor products. All chain complexes are homologically graded, with the differential of degree $-1$; in particular, cohomology of a topological space is placed in \emph{non-positive} degrees.  Homology and cohomology where coefficients are not specified is always assumed to be computed with rational coefficients. For an oriented manifold $M$, we denote by $\mathrm{PD}\colon H_c^\bullet(M)\to H_{\dim(M)+\bullet}(M)$ the Poincaré duality isomorphism (note that the plus sign in $\dim(M)+\bullet$ comes from the abovementioned homological degree convention).

To handle suspensions of chain complexes, we use the formal symbol $s$ of degree $1$, and let $sC_\bullet=\k s\otimes C_\bullet$; similarly, the symbol $s^{-1}$ implements desuspensions. The suspension operad $\mathcal S$ and the desuspension operad $\mathcal S^{-1}$ as the linear species with components given by the respective formulas 
 \[
\mathcal S(n)=\Hom_\k((\k s)^{\otimes n}, \k s)\text{ and }\mathcal S^{-1}(n)=\Hom_\k((\k s^{-1})^{\otimes n}, \k s^{-1})    
 \]
with the operad structures given by substitution of multilinear maps into each other. The operadic suspension and desuspension are then given by the Hadamard product with $\mathcal S$ and $\mathcal S^{-1}$ respectively; note that the directions of degree shifts in operadic (de)suspension are opposite to those of (de)suspensions of chain complexes.

\subsection*{Conflicts of notation} Some of the notation of this paper clashes with various choices of notation present in the literature. First of all, there seems to be no uniform choice of notation for generators of Chow rings of De Concini--Procesi wonderful models. We choose to follow the convention of Pagaria--Pezzoli \cite{MR4675068} where the ``De Concini--Procesi generators'' are denoted by $x_g$ and the ``Yuzvinsky generators'' are denoted by $\sigma_g$ (some history of these generators is given in Section~\ref{sec:rec} below). Furthermore, Westerland in \cite{MR3068956} defines the $d$-dimensional version of the gravity operad of Getzler \cite{MR1284793,MR1363058} which he denotes $\Grav_d$, and then identifies defines an unnamed (though arguably more important) operad by removing from $\Grav_d$ extra unary operations. We decided to denote (a $2d$-fold suspension of) that latter operad $\widetilde\Grav_d$, and to denote the Koszul dual of the desuspension of $\widetilde\Grav_d$ by $\widetilde\HyperCom_d$, this way not departing extremely far from the notation of Getzler for $d=1$. Furthermore, we chose to denote the homology of the operad of Chen--Gibney--Krashen spaces by $\HyperCom_d$; thus, overall, the operads $\HyperCom_d$ and $\Grav_d$ are not really related in very direct way for $d>1$. It is also worth noting that two of the authors of the present paper and Tamaroff defined in~\cite{MR4580527}, for each integer $k\ge 1$, the operad $k\text{-}\HyperCom$ as the operad representing homotopy quotient by the circle action of the operad of commutative associative algebras with a circle action that acts as an operator of differential order $k+1$. From the point of view of this article, one should perhaps use the notation $\frac{d+1}{d}\text{-}\HyperCom$ for the operad $\HyperCom_d$, but we chose to emphasize the complex dimension of the question rather than push the philosophy of the differential order this far.  

\subsection*{Acknowledgements}

This work was supported by the ANR project HighAGT (ANR-20-CE40-0016, V.D. and G.S.), by the Math-AmSud project HHMA (23-MATH-06, V.D. and E.H.), by the Dutch Research Council project OCENW.M.21.233 (S.S.), by a DFG Walter Benjamin Fellowship (project number 561158824, G.S.), by University of Strasbourg Institute for Advanced Studies (E.H.) and by Institut Universitaire de France (V.D.).
We thank Basile Coron, Imma Gálvez Carrillo, Joana Cirici, Mikhail Kapranov, Sergei Merkulov, Andrew Tonks, Bruno Vallette, and Alexander Voronov for useful discussions of various results of the paper. Special thanks are due to Clément Dupont for pointing a gap in the proof of Proposition \ref{prop:split}, as well as for explaining how to fix that gap, and to Thomas Willwacher for detailed discussions of formality and non-formality of various operads. Finally, we would like to acknowledge the extraordinary impact that work of Yuri Ivanovich Manin had on the development of several research areas where this paper belongs, and dedicate our work to his memory.

\section{Recollections}\label{sec:rec}

\subsection{Operads and operadic Gröbner bases}
We refer the reader to the monograph \cite{LV} for a systematic treatment of operads. In discussing operads, we use the language of combinatorial species \cite{MR2724388,MR1629341}. In particular, we mostly think of an operad on a species $\calO$ in terms of partial compositions \cite[Sec.~5.3.7]{LV}
 \[
\circ_\star\colon\calO(I\sqcup\{\star\})\otimes\calO(J)\to\calO(I\sqcup J),     
 \]
which, if one uses the notion of the derivative species $\partial(\calO)$ and the Cauchy product of species $\otimes$, assemble into a map
 \[
\circ_\star\colon\partial(\calO)\otimes\calO\to\calO.
 \]
The ``sequential'' and the ``parallel'' axioms satisfied by the parallel compositions are given by the commutative diagrams
 \[
 \xymatrix@M=6pt{
\partial(\calO)\otimes\partial(\calO)\otimes\calO\ar@{->}^{\mathrm{id}\otimes \circ_\star}[rr] \ar@{->}_{\mu\otimes \mathrm{id}}[d] & & \partial(\calO)\otimes\calO \ar@{->}^{\circ_\star}[d]  \\ 
\partial(\calO)\otimes\calO\ar@{->}_{\circ_\star}[rr]  & & \calO      
 } 
 \]
and  
 \[
 \xymatrix@M=6pt{
\partial(\partial(\calO))\otimes\calO\otimes\calO\ar@{->}^{(\rho\otimes\mathrm{id})(\mathrm{id}\otimes \sigma) }[rr] \ar@{->}_{\rho\otimes \mathrm{id}}[d] & & \partial(\calO)\otimes\calO \ar@{->}^{\circ_{\star_1}}[d]  \\ 
\partial(\calO)\otimes\calO\ar@{->}_{\circ_{\star_2}}[rr]  & & \calO      
 } .
 \]
Here the structure morphisms
 \[
\mu\colon\partial(\calO)\otimes\partial(\calO)\to\partial(\calO) \quad\text{ and }\quad \rho\colon\partial(\partial(\calO))\otimes\calO\to\partial(\calO)
 \]
are obtained by applying $\partial$ to $\circ_\star$, and $\sigma\colon\calO\otimes\calO\cong\calO\otimes\calO$ is the symmetry isomorphism of the Cauchy product.

Some of our arguments rely in a very substantial way on the theory of Gröbner bases for operads \cite{BD,MR2667136}. To assist the reader who is not fluent in the language of Gröbner bases, we summarize the fundamental features of that theory. In general, a theory of Gröbner bases is available when one considers an algebraic structure where free objects are ``combinatorial'': each free object has a $\k$-basis of ``monomials'', and the result of applying any structure operation to monomials is again a monomial. An ordering of monomials is said to be admissible if it is a total well ordering, and if each structure operation is an increasing function of its arguments: replacing one of the monomials to which one applies that structure operation by a greater one increases the result. Given an admissible ordering, one can associate to any ideal $I$ of the free object the ``ideal of leading terms'' $LT(I)$ which is spanned by leading terms of elements of $I$ with respect to the given ordering. 
The importance of that notion is that the cosets of monomials that do not belong to $LT(I)$ form a basis in the quotient by the ideal $I$. 
A Gröbner basis of an ideal $I$ is a subset $G\subset I$ whose leading terms generate the ideal $LT(I)$; knowing a Gröbner basis of an ideal $I$ is thus instrumental in efficiently working with the quotient by that ideal. Note that this formalism does not quite apply in the case for arbitrary symmetric operads (where structure operations are compositions along arbitrary trees), but there exist a notion of a ``shuffle operad'' which forgets the symmetric group action on an operad, retaining however all the essential structure, and free shuffle operads admit monomial bases.

\subsection{Wonderful models of subspace arrangements and their operad-like structures}\label{sec:dcp}

Let us recall the basics of the general theory of wonderful models of subspace arrangements was developed by De Concini and Procesi~\cite{MR1366622}. 

Suppose that $V$ is a vector space. Let $\calL$ be a collection of subspaces of $V$ closed under intersections, and let $\calC$ be a collection of subspaces of $V^*$ that are equations of the elements of $\calL$, so that each $H\in\calC$ corresponds to 
 \[
H^\bot=\{x\in V\colon f(x)=0 \text{ for all } f\in H\}\in\calL.     
 \]
Recall that a set $\calG\subseteq\calC$ is said to be a \emph{building set} for $\calL$ if every element $X\in\calC$ is the direct sum of the maximal elements of $\calG$ that it contains. We shall assume that $\calG$ contains $V^*$. Once we fix a building set $\calG$, we can define a \emph{nested set} relative to $\calG$ as a subset $\{G_1,\ldots,G_k\}\subseteq\calG$ such that for all subsets $\{G_{i_1},\ldots,G_{i_m}\}$ of pairwise non-comparable (with respect to inclusion) elements, then  $G_{i_1}+\cdots+G_{i_m}\notin\calG$, and the sum $G_{i_1}+\cdots+G_{i_m}$ is direct. This definition gives rise to the definition of the \emph{nested set complex} $N(\calL,\calG)$, the simplicial complex with vertices $\calG$ whose faces are nested sets relative to $\calG$. 

Given a building set $\calG$ of a subspace arrangement $\calL$, one may define a \emph{wonderful model} $Y_{\calL,\calG}$ as the closure in 
 \[
\prod_{G\in\calG} \mathbb{P}(V/G^\bot)     
 \]
of $\mathbb{P}(V\setminus\bigcup_{G\in \calG} G^\bot)$. De Concini and Procesi \cite{MR1366622} proved that $Y_{\calL,\calG}$ is smooth, with the complement of $\mathbb{P}(V\setminus\bigcup_{G\in \calG} G^\bot)$ being a normal crossing divisor whose irreducible components $D_G$ correspond to elements $G\in\calG\setminus \{V^*\}$. They also computed the cohomology ring of $Y_{\calL,\calG}$, for an  arrangement of subspaces $\calL$ in a vector space $V$ over complex numbers. The answer is described as follows. Recall that to any simplicial complex $\Gamma$ on $m$ vertices one can associate \cite{MR725505} a commutative ring called the \emph{Stanley--Reisner ring} and denoted $SR(\Gamma)$. It is defined by the formula
 \[
SR(\Gamma)=\mathbb{Z}[x_1,\ldots,x_m]/(x_{i_1}\cdots x_{i_k}\colon \{i_1,\ldots,i_k\}\notin\Gamma).     
 \]  
In particular, we have
 \[
SR(N(\calL,\calG))=\mathbb{Z}[x_G\colon G\in\calG]/(x_{G_1}\cdots x_{G_k}\colon \{G_1,\ldots,G_k\}\notin N(\calL,\calG)).     
 \] 
In \cite[Th.~5.2]{MR1366622}, De Concini and Procesi establish a ring isomorphism
 \[
H^\bullet(Y_{\calL,\calG},\mathbb{Z})\cong SR(N(\calL,\calG))/\left(\prod_{i=1}^kx_{G_i}\bigl(\sum_{H\supseteq G}x_H\bigr)^{d_{G_1,\ldots,G_k,G}} \colon G_i\subsetneq G, \{G_1,\ldots,G_k\}\in N(\calL,\calG)\right) ,    
 \]
where $d_{G_1,\ldots,G_k,G}=\dim G-\dim \sum_{i=1}^k G_i$. It is worth noting that the generators $x_G$ for $G\ne V^*$ are the Poincaré duals of the divisors $D_G$, and the generator $x_{V^*}$ is the pullback of the hyperplane class of $\mathbb{P}(V)=\mathbb{P}(V/(V^*)^\bot)$.

It will also be useful for us to consider another set of generators of the cohomology rings, namely the generators $\sigma_G$, $G\subseteq\calG$, defined by the formula 
 \[ 
\sigma_G=\sum_{H\supseteq G} x_H.      
 \]
These elements appear already in the original work of De Concini and Procesi \cite[p.~486]{MR1366622}, where it is indicated that $\sigma_G$ is the pullback of the hyperplane class in $H^2(\mathbb{P}(V/G^\bot))$ under the canonical map 
 \[
Y_{\calL,\calG} \to \mathbb{P}(V/G^\bot).    
 \] 
It seems that the importance of viewing these elements as a different choice of generators goes back to the work of Yuzvinsky \cite[Sec.~3]{MR1881024}. (More recently, they have been studied for Chow rings of matroids by Backman, Eur and Simpson \cite{MR4780488} under the name ``simplicial generators'', and for Chow rings of polymatroids by Pagaria and Pezzoli in \cite{MR4675068}.) It is known that under that  change of generators, we have
 \[
SR_\sigma(N(\calL,\calG))\cong\mathbb{Z}[\sigma_G\colon G\in\calG]/((\sigma_{G_1}-\sigma_G)\cdots (\sigma_{G_k}-\sigma_G)\colon \{G_1,\ldots,G_k\}\notin N(\calL,\calG), G=G_1+\cdots+G_k),      
 \]
and 
 \[
H^\bullet(Y_{\calL,\calG},\mathbb{Z})\cong SR_\sigma(N(\calL,\calG))/\left(\prod_{i=1}^k(\sigma_{G_i}-\sigma_G)\sigma_G^{d_{G_1,\ldots,G_k,G}} \colon  G_i\subsetneq G, \{G_1,\ldots,G_k\}\in N(\calL,\calG)\right) .    
 \]

An important ingredient in the arguments of De Concini and Procesi is the fact that the boundary of each wonderful model is made of products of smaller wonderful models \cite[Th.~1.4]{MR3687430}. It was first observed by Rains \cite{MR2746338} that this gives an operad-like structure on all wonderful models considered at the same time. For us, a slightly different viewpoint will prove advantageous, namely that of the Feynman category $\mathfrak{LBS}$ of \emph{built lattices} of Coron \cite{MR4948093}. Within this formalism, a generalized operad $\calO$ has components $\calO(\calL,\calG)$, where the ``arity'' $(\calL,\calG)$ is a pair consisting of a geometric lattice $\calL$ and its building set $\calG$ \cite[Sec.~2]{MR4948093}. Infinitesimal compositions 
 \[
\circ_\star\colon\calO(I\sqcup\{\star\})\times\calO(J)\to\calO(I\sqcup J)    
 \]
of an operad are generalized to the structure operations 
 \[
\calO(\calL_G,\calG_G)\times \calO(\calL^G,\calG^G)\to\calO(\calL,\calG),     
 \]
that should be available for every $G\in\calG$. Here we use the arrangements 
 \[
\calL_G:=\{H^\bot\in\calL\colon G\subseteq H\}, \quad \calL^G:=\{H^\bot\in\calL\colon H\subseteq G\}   
 \]
and their building sets
 \[
\calG_G:=\{H+G\colon H\in\calG, H\not\subseteq G\}, \quad \calG^G:=\{H \in\calG\colon H\subseteq G\},     
 \]
which are particular cases of induced building sets \cite[Def.~2.10]{MR4948093}. These maps satisfy certain relations \cite[Prop.~3.19]{MR4948093} that generalize the usual relations satisfied by infinitesimal compositions in an operad. (Strictly speaking, for arbitrary arrangements it is not enough to consider geometric lattices, and one needs a generalization of the formalism of Coron to polymatroids \cite[Sec.~7.2]{MR4948093}, but throughout this paper all lattices we consider turn out to be geometric, so we may allow ourselves to not over-complicate the situation.)

Let us record a simple result which is at the heart of the definitions of \cite{MR4948093}, though it is not stated or proved as such.
\begin{proposition}
For every $G\in\calG$, there is a map 
 \[
\prod_{H\in\calG_G} \mathbb{P}(V/H^\bot) \times      
\prod_{H\in\calG^G} \mathbb{P}(G^\bot/H^\bot)   
\to    
\prod_{H\in\calG} \mathbb{P}(V/H^\bot)    
 \]
which restricts to a map  
\begin{equation}\label{eq:operad-DCP}
Y_{\calL_G,\calG_G}\times Y_{\calL^G,\calG^G}\to Y_{\calL,\calG}.     
\end{equation}
These maps satisfy the axioms of an $\mathfrak{LBS}$-operad.  
\end{proposition}

\begin{proof}
To define a map into the product 
 \[
\prod_{H\in\calG} \mathbb{P}(V/H^\bot),         
 \]
we should define maps into each factor. There are two types of factors: those for which $H$ is contained in $G$, and those for which $H$ is not contained in $G$. For the first one, we shall use the map from the corresponding factor in $\prod_{H\in\calG^G} \mathbb{P}(G^\bot/H^\bot)$ induced by the map $G^\bot/H^\bot\hookrightarrow V/H^\bot$. For the second one, we shall use the map from the factor $\mathbb{P}(V/(H+G)^\bot)$ in $\prod_{H\in\calG^G} \mathbb{P}(G^\bot/H^\bot)$ induced by the map $V/(H+G)^\bot\to V/H^\bot$. By direct inspection, these restrict to wonderful compactifications and satisfy the relevant axioms.
\end{proof}

Examining the above proof and the proof of \cite[Th.~3.2]{MR1366622}, we immediately see that the product $Y_{\calL_G,\calG_G}\times Y_{\calL^G,\calG^G}$ is precisely the irreducible component $D_G$ of the boundary divisor of the wonderful compactification. This allows us to describe the operations induced by the maps \eqref{eq:operad-DCP} on the cohomology. 

\begin{proposition}\label{prop:operad-on-cohomology}
The map
 \[
Y_{\calL_G,\calG_G}\times Y_{\calL^G,\calG^G}\to Y_{\calL,\calG}     
 \]
combined with the Poincaré duality isomorphism, induces the operation on the cohomology
 \[
H^\bullet(Y_{\calL_G,\calG_G})\otimes H^\bullet(Y_{\calL^G,\calG^G})\to 
H^{\bullet-2}(Y_{\calL,\calG})
 \]
obtained by composition of the algebra homomorphism 
 \[
H^\bullet(Y_{\calL_G,\calG_G})\otimes H^\bullet(Y_{\calL^G,\calG^G})\to H^{\bullet}(Y_{\calL,\calG})/\mathrm{Ann}(x_G)     
 \] 
defined by
\begin{align*}
\sigma_{H+G}\otimes 1&\mapsto \sigma_H\\
1\otimes \sigma_H&\mapsto \sigma_H,
\end{align*}
and the obvious identification $H^{\bullet}(Y_{\calL,\calG})/\mathrm{Ann}(x_G)\cong x_G H^{\bullet}(Y_{\calL,\calG})\subseteq H^{\bullet-2}(Y_{\calL,\calG})$.
\end{proposition}

\begin{proof}
The map \eqref{eq:operad-DCP} induces, by virtue of the Künneth formula, a map composition
 \[
H_{\bullet}(Y_{\calL_G,\calG_G})\otimes H_{\bullet}(Y_{\calL^G,\calG^G})\cong
H_{\bullet}(Y_{\calL_G,\calG_G}\times Y_{\calL^G,\calG^G})\to
H_{\bullet}(Y_{\calL,\calG}).     
 \]
Using the Poincaré duality of the cohomology algebras, one can view these maps as maps 
 \[
H^{-2\dim(Y_{\calL_G,\calG_G})+\bullet}(Y_{\calL_G,\calG_G})\otimes H^{-2\dim(Y_{\calL^G,\calG^G})+\bullet}(Y_{\calL^G,\calG^G})\to
H^{-2\dim(Y_{\calL,\calG})+\bullet}(Y_{\calL,\calG}).     
 \]
Their description given above follows from the fact that the map of homology induced by the inclusion of the component $D_G$ of boundary divisor gets identified by the Poincaré duality isomorphism with the intersection product with $x_G$, and from the description of the elements $\sigma_H$ as the pullbacks of the hyperplane class of the respective projective spaces. 
\end{proof}

\begin{remark}
The same formulas are used to define an \emph{ad hoc} operad-like structure for arbitrary Chow rings of polymatroids in \cite[Lemma 4.14]{MR4675068}; as we indicated above, this fits the generalization of the formalism of Coron to polymatroids \cite[Sec.~7.2]{MR4948093}.  
\end{remark}

\section{The Deligne--Mumford spaces and the Chen--Gibney--Krashen spaces}\label{sec:DM-CGK}

\subsection{The braid arrangement}

The arrangement of subspaces that is of utmost importance for us is the \emph{braid arrangment}. It is a subspace arrangement in the vector space $V_n=\mathbb{C}^n/\mathbb{C}(1,1,\ldots,1)$ which we shall denote by $\calL(\Delta_n)$, whose collection of dual subspaces is
 \[
\calC(\Delta_n)=
\left\{  
\bigoplus_{\pi \text{ a partition of } \{1,\ldots,n \}}\mathrm{span}\{z_i-z_j : i\sim_\pi j \}
\right\} .    
 \]
In plain words, $\calL(\Delta_n)$ is the collection of all possible intersections of diagonal hyperplanes $z_i=z_j$ in the quotient $\mathbb{C}^n/\mathbb{C}(1,1,\ldots,1)$. Note that if we consider the projectivization of the complement $\mathbb{P}(V_n\setminus \calL(\Delta_n))$, it can be identified with the space of $n$-tuples of complex numbers modulo the action of the group $z\mapsto az+b$ of affine transformations, which is the same as $\calM_{0,n+1}$, the space of $n$-tuples of points of $\mathbb{C}P^1$ modulo projective transformations. 

The minimal building set $\calG_{\min}(\Delta_n)$ of $\calL(\Delta_n)$ consisting of all subspaces 
 \[
H_I=\{z_i=z_j \text{ for } i,j\in I \} \text{ for } I\subsetneq \{1,\ldots,n \} ,  |I|\geq 2.  
 \]
It is well known that the wonderful compactification $Y_{\calL(\Delta_n),\calG_{\min}(\Delta_n)}$ is isomorphic to the genus zero Deligne--Mumford space $\overline{\calM}_{0,n+1}$. 

Let us explain what the maps \eqref{eq:operad-DCP} give in this case. For $I\subsetneq \{1,\ldots,n \}$,  $|I|\geq 2$, we have 
$\calL_I(\Delta_n)$ consists of the intersections that refine the diagonal $H_I$, so we may regard all elements from $I$ as one single element, which we denote $\star$. This allows us to identify $\calL_I(\Delta_n)\cong \calL(\Delta_{I^c\sqcup\{\star\}})$. Clearly, we have $\calL^I(\Delta_n)\cong \calL(\Delta_{I})$, so the map \eqref{eq:operad-DCP} becomes  
 \[
Y_{\calL(\Delta_{I^c\sqcup\{\star\}}),\calG_{\min}(\Delta_{I^c\sqcup\{\star\}})}\times Y_{\calL(\Delta_I),\calG_{\min}(\Delta_I)}\to Y_{\calL(\Delta_n),\calG_{\min}(\Delta_n)} ,
 \]
which corresponds to the usual operad structure on $\{\overline{\calM}_{0,n+1}\}$. 

\subsection{Stable rooted trees of projective spaces}

Let us now consider the space of $n$-tuples of vectors of $\mathbb{C}^d$ modulo the action of the group $z\mapsto az+b$ of translations and homotheties. That space was considered by Chen, Gibney and Krashen \cite{MR2496455} who denoted it $TH_{d,n}$, and constructed a particular compactification of it, which they denoted $T_{d,n}$. They described the spaces $T_{d,n}$ as moduli spaces of  stable rooted trees of $d$-dimensional projective spaces with $n$ marked points. We shall now explain what this means. We adopt a different notational convention which we find more suggestive: we denote $\CGK_d(n):=TH_{d,n}$ and $\bCGK_d(n):=T_{d,n}$.

Suppose that $\tau$ is a usual combinatorial rooted tree with $n$ leaves, with edges directed towards the root; we assume that $\tau$ is \emph{stable}, so that each vertex has at least two incoming edges (note that contrary to the classical graph theory, we do not regard the leaves as vertices of a tree). A $\tau$-shaped stable rooted tree of $d$-dimensional projective spaces is a connected variety constructed in the following way. One takes, for each vertex $v$ of $\tau$, a variety $X_v$ isomorphic to the $d$-dimensional projective space and  equipped with a closed embedding $f_v\colon\mathbb{P}^{d-1}\to X_v$ (corresponding, in a sense, to the outgoing edge of $v$). We also require a choice of marked points outside the image of $f_v$, which are in bijection with the incoming edges of the corresponding vertex $v$ (thus, each component has at least two marked points). The varieties $X_v$ are glued together via blowups: to glue the $X_v$ corresponding to the vertex $v$ to the component $X_w$ corresponding to the vertex $w$ such that there is a directed edge $e\colon v\to w$ in $\tau$, one blows up the component $X_w$ at the marked point corresponding to $e$, and identifies the divisor $f_v(\mathbb{P}^{d-1})$ on $X_v$ with the exceptional divisor of the blowup. Note that the open part corresponds to the corolla (a single-vertex tree with $n$ leaves) and is naturally identified with $\CGK_d(n)$ (a configuration of points in $\mathbb{P}^{d}$ outside a hyperplane is, up to isomorphism preserving the hyperplane, the same as configuration of points in an affine space modulo translations and homotheties). 

Even though  the word ``operad'' does not appear in \cite{MR2496455}, their construction is very much of operadic nature. Indeed, by the very nature of the definition, the collection $\bCGK_d$ for a fixed $d$ is an algebra over the monad of rooted trees, hence an operad \cite{LV}. We can then compute the (say, rational) homology of this operad, obtaining an operad in chain complexes $H_\bullet(\bCGK_d)$. One of the main results of this paper is an explicit description of this operad. 

\begin{remark}
A reader with a particular passion for Chen--Gibney--Krashen spaces may be interested to know that the conjecture about pairing of cycles posed in \cite[Sec.~6.1]{MR2496455} is now a theorem: once one views the spaces $\bCGK_d(n)$ as wonderful models of subspace arrangements, this conjecture is a direct consequence of a much more general result of Pagaria--Pezzoli \cite[Lemma 4.4]{MR4675068}.
\end{remark}

\subsection{Multiples of an arrangement}

To place the Chen--Gibney--Krashen spaces in the context of wonderful models, it will be convenient for us to discuss a general operation on subspace arrangements which we shall call ``multiplication by $d$''. Namely, if $\calL$ is an arrangement of subspaces in a vector space $V$ with the corresponding to it collection of subspaces $\calC$ in $V^*$, we can define the collection $\calC^{(d)}$ of subspaces in $(V^*)^{\oplus d}$ by setting
 \[
\calC^{(d)}:=\{ X^{\oplus d}\colon X\in\calC \} ,
 \]
and the subspace arrangement $\calL^{(d)}$ in $V^{\oplus d}$ by setting 
 \[
\calL^{(d)}:=\{ U^\bot\colon U\in \calC^{(d)} \}.     
 \] 
This notion goes back to von Neumann \cite{Neumann}. 

\begin{remark}
If $V$ is a vector space over $\mathbb{R}$ and $d=2$, this construction can be viewed as the complexification of the collection of subspaces, which explains why it is common to use $c$ instead of $d$ in the formulas, and call the process ``$c$-plexification'', capitalizing on the phonetic similarity to ``complexification''.
\end{remark}

If $\calG$ is a building set for $\calL$, then
 \[
\calG^{(d)} :=\{ G^{\oplus d}\colon G\in\calG \} 
 \]
is easily seen to be a building set of $\calL^{(d)}$. Additionally, there is an obvious map $\imath_{d,d'}\colon V^{\oplus d}\to V^{\oplus d'}$ for $d<d'$ that sends $v$ to $v\oplus 0^{\oplus (d'-d)}$ which satisfies $\imath_{d,d'}^*(\calG^{(d')})=\calG^{(d)}$. This map induces a map 
 \[
\prod_{G\in\calG^{(d)}} \mathbb{P}(V^{\oplus d}/G^\bot) \to     
\prod_{G\in\calG^{(d')}} \mathbb{P}(V^{\oplus d'}/G^\bot)     
 \]
that restricts to a well defined map
 \[
Y_{\calL^{(d)},\calG^{(d)}} \to Y_{\calL^{(d')},\calG^{(d')}} .     
 \] 
By a direct inspection, one sees that this map is compatible with the presentations of the cohomology discussed above; specifically, it induces the map
 \[
H^\bullet(Y_{\calL^{(d')},\calG^{(d')}},\mathbb{Z})\to H^\bullet(Y_{\calL^{(d)},\calG^{(d)}},\mathbb{Z})      
 \]
sending each generator $x_G$ to the same generator $x_G$ (or, equivalently, each generator $\sigma_G$ to the same generator $\sigma_G$).

If we consider the braid arrangement $\calL(\Delta_n)$ and its minimal building set $\calG_{\min}(\Delta_n)$, multiplication by $d$ produces a subspace arrangement $\calL^{(d)}(\Delta_n)$ and its minimal building set $\calG^{(d)}_{\min}(\Delta_n)$. In this case, $\mathbb{P}(V_n^{\oplus d}\setminus \calL^{(d)}(\Delta_n))$ is naturally identified with $\CGK_d(n)$, and, moreover, Gallardo and Routis \cite{MR3687430} proved that there is an isomorphism 
 \[
Y_{\calL^{(d)}(\Delta_n),\calG^{(d)}_{\min}(\Delta_n)}\cong \bCGK_d(n).     
 \]
The lattice of intersections of the subspace arrangement $\calL^{(d)}(\Delta_n)$ is the same as the lattice of intersections of the hyperplane arrangement $\calL(\Delta_n)$, that is the partition lattice. It follows that for each fixed $d\ge 1$, there is an operad structure on the collection of spaces $\{Y_{\calL^{(d)}(\Delta_n),\calG^{(d)}_{\min}(\Delta_n)}\}$. Similarly to how one proves that the operad $\{\overline{\calM}_{0,n+1}\}$ is isomorphic to $\bCGK_1$, one can show that the operad structure on $\bCGK_d$ obtained via the maps \eqref{eq:operad-DCP} coincides with the operad structure coming from the rooted tree nature of the Chen--Gibney--Krashen construction.

Chen, Gibney and Krashen give in \cite[Sec.~6]{MR2496455} a presentation of the Chow ring of $\bCGK_d(n)$, and further establish in \cite[Sec.~7.3]{MR2496455} that the Chow ring $A^\bullet(\bCGK_d(n))$ is isomorphic to the integral cohomology ring $H^{\bullet}(\bCGK_d(n),\mathbb{Z})$, with the isomorphism multiplying degrees by two. Their argument uses the results of Fulton and MacPherson \cite{MR1259368}, but since we know that these spaces are in fact wonderful compactifications, we can deduce the same result from the De Concini--Procesi presentation. In fact, it is a consequence of the following general statement.

\begin{proposition}
Let $\calL$ be an arrangement of hyperplanes. We have
 \[
H^\bullet(Y_{\calL^{(d)},\calG^{(d)}},\mathbb{Z})\cong SR(N(\calL,\calG))/\left(\bigl(\sum_{H\supseteq G}x_G\bigr)^d \colon G\in\calG, \dim G=1\right).    
 \]
\end{proposition}

\begin{proof}
For $d=1$, this simplification of the De Concini--Procesi presentation is well known, see, e.g. \cite[Th.~1]{MR2038195}. Moreover, for any hyperplane arrangement $\calL$, the combinatorics of the arrangement $\calL^{(d)}$ does not depend on $d$, and the argument of \cite[Th.~1]{MR2038195} showing that all the generators of the ideal of relations of \cite[Th.~5.2]{MR1366622} belong to the ideal generated by the given relations adapts \emph{mutatis mutandis}, since it only uses the fact that the Stanley--Reisner relations imply that some terms in the product of $\sum_{C\supseteq B}x_C$ and some monomials in the variables $x_G$ vanish, and the same will hold for the power $(\sum_{H\supseteq G}x_H)^d$. 
\end{proof}

Therefore, recalling that the minimal building set of the lattice of partitions of $\underline{n}$ is in a bijection with the subsets of $\underline{n}$ of cardinality at least two and making the relations of the Stanley--Reisner ring explicit, we have 
 \[
H^\bullet(\bCGK_d(n),\mathbb{Z}) \cong \frac{\mathbb{Z}[x_I\colon I\subseteq\underline{n}, |I|\ge 2]}
{(\{x_Ix_J\colon I\cap J\ne\varnothing, I\not\subseteq J, J\not\subseteq I\}, \{(\sum_{\{i,j\}\subseteq I}x_I)^d\colon i,j\in\underline{n} \} )}.    
 \]
Note that the generators $x_G$ are denoted by $\delta_G$ by Chen, Gibney, and Krashen \cite[Sec.~6]{MR2496455}.  The generators that we denote $\sigma_G$ also appear in \cite{MR2496455} (under the name $\eta_G$) as instances of very ample divisor classes on $\bCGK_d(n)$. The presentations in terms of those generators is 
 \[
H^\bullet(\bCGK_d(n))\cong 
\frac{\mathbb{Z}[\sigma_I\colon I\subseteq\underline{n}, |I|\ge 2]}
{(\{(\sigma_I-\sigma_{I\cup J})(\sigma_J-\sigma_{I\cup J})\colon I\cap J\ne\varnothing, I\not\subseteq J, J\not\subseteq I\}, \{\sigma_{\{i,j\}}^d\colon i\ne j\in\underline{n} \} )}.     
 \]
Moreover, the operad-like structure from Proposition \ref{prop:operad-on-cohomology} is given by the maps

 \[
H^\bullet(\bCGK_d(I^c\sqcup\{\star\}))\otimes H^\bullet(\bCGK_d(I))\to 
H^{\bullet-2}(\bCGK_d(n))
 \]
obtained by composition of the algebra homomorphism 
 \[
H^\bullet(\bCGK_d(I^c\sqcup\{\star\}))\otimes H^\bullet(\bCGK_d(I))\to H^{\bullet}(\bCGK_d(n))/\mathrm{Ann}(x_G)   
 \] 
defined by
\begin{align*}
\sigma_K\otimes 1&\mapsto 
\begin{cases}
\quad\,\,\, \sigma_K, \qquad\, \star\notin K,\\
\sigma_{K\sqcup I\setminus \{\star\}}, \quad \star\in K,
\end{cases} \\    
1\otimes \sigma_K&\mapsto \sigma_K,
\end{align*}
and the obvious identification 
 \[
H^{\bullet}(\bCGK_d(n))/\mathrm{Ann}(x_G)\cong x_G H^{\bullet}(\bCGK_d(n))\subseteq H^{\bullet-2}(\bCGK_d(n)).    
 \]

\section{The gravity operad and its properties}\label{sec:gravity}

\subsection{The operad of Westerland}\label{sec:grav-op}

In \cite{MR3068956}, Westerland made a major advance in studying the open parts $\CGK_d(n)$ of the Chen--Gibney--Krashen spaces. Namely, he described the operad structure on the collection $H_{\bullet-1}(\CGK_d)$ as a homological transfer map (this essentially goes back to Kimura, Stasheff and Voronov \cite{MR1341693}), and found an algebraic description of that operad. Let us recall that construction in some detail following Dupont and Horel \cite[Sec.~4]{MR3767344}. One begins by considering the spaces
 \[
C_{2d}(n)=\Conf(n,\mathbb{R}^{2d})/(\mathbb{R}^{2d}\rtimes \mathbb{R}_{>0}),
 \] 
that is the quotients of the configuration spaces of $n$ distinct labelled points in $\mathbb{R}^{2d}$ by translations and dilations. We note that there is an obvious homeomorphism
 \[
\Conf(n,\mathbb{R}^{2d})\cong \Conf(n,\mathbb{C}^{d}),     
 \]
so there is a natural $S^1$-bundle
 \[
C_{2d}(n)\to \CGK_d(n).     
 \] 
If one considers the Fulton--MacPherson compactification $C_{2d}(n)\hookrightarrow FM_{2d}(n)$, it turns out that the $S^1$-action on $C_{2d}(n)$ extends to $FM_{2d}(n)$ and is compatible with the stratification of $FM_{2d}(n)$, so that the quotient $X_{2d}(n):=FM_{2d}(n)/S^1$ also has a stratification $X_{2d}(\tau)$ indexed by stable rooted trees $\tau$ with $n$ leaves, and there is a commutative square 
 \[\begin{tikzcd}
C_{2d}(n)\arrow{r}{\sim} \arrow[swap]{d}{} & FM_{2d}(n) \arrow{d}{} \\
\CGK_d(n)\arrow{r}{\sim} & X_{2d}(n)
\end{tikzcd}
 \]
where the horizontal arrows are embeddings that are additionally homotopy equivalences and the vertical arrows are the quotient maps by the $S^1$-action. Furthermore, for a two-vertex tree $\tau$ with $J$ the set of leaves not adjacent to the root, there is an $S^1$-bundle
 \[
X_{2d}(\tau)\to X_{2d}(I\sqcup\{\star\})\times X_{2d}(J),     
 \]
leading to a map in the homology
 \[
H_\bullet(X_{2d}(I\sqcup\{\star\})\times X_{2d}(J))\to H_{\bullet+1}(X_{2d}(\tau))\to  H_{\bullet+1}(X_{2d}(I\sqcup J))   
 \] 
which, via the Künneth formula, induces an operad structure on 
 \[
\{H_{\bullet-1}(X_{2d}(n))\}\cong\{H_{\bullet-1}(\CGK_d(n))\}. 
 \]
In our work, we shall consider the operad 
 \[
\widetilde\Grav_d:=\mathcal S^{2d}H_{\bullet-1}(\CGK_d)   
 \]
obtained from the operad of Westerland by a $2d$-fold operadic suspension. Algebraically, this amounts to even degree shifts
 \[
\mathcal S^{2d}H_{\bullet-1}(\CGK_d(n))=H_{2d(n-1)+\bullet-1}(\CGK_d(n)),     
 \]
which are somewhat ``cosmetic'', and in particular have no impact on signs in formulas. However, besides proving algebraically convenient in the following sections, this suspension has a conceptual meaning behind it. Indeed, since $\dim\CGK_d(n)=2d(n-1)-2$, we have the Poincaré duality isomorphism
 \[
H^{\bullet+1}_c(\CGK_d(n))\cong H_{2d(n-1)+\bullet-1}(\CGK_d(n)).     
 \] 
The corresponding operad structure on $\{H^{\bullet+1}_c(\CGK_d(n))\}$ can be also described directly as follows. The Chen--Gibney--Krashen construction leads to a stratification of $\bCGK_d(n)$ by spaces $\bCGK_d(\tau)$ corresponding to stable rooted trees $\tau$ with $n$ leaves, and the stratum corresponding to the two-vertex tree with $J\subset\underline{n}$ being the set of leaves attached to the only non-root vertex is naturally isomorphic to
 \[
\CGK_d(I\sqcup\{\star\})\times \CGK_d(J).     
 \]
This stratum is adjacent to $\CGK_d(I\sqcup J)$, and one can consider the long exact sequence of a pair, giving the connecting homomorphism
 \[
H^{\bullet}_c(\CGK_d(I\sqcup\{\star\})\times \CGK_d(J))) \to H^{\bullet-1}_c(\CGK_d(I\sqcup J)),   
 \] 
which is precisely the Poincaré dual of the operad structure on $\widetilde\Grav_d=\mathcal S^{2d}H_{\bullet-1}(\CGK_d)$. 

Westerland's algebraic description of this operad uses the operad that he denotes $\Grav_d$, generated by symmetric elements $g_I\in\Grav_d(I)$ of homological degree $2d-1$ and $c\in\Grav_d(1)$ of homological degree $-2$ subject to the following relations:
\begin{align*}
\sum_{\{i,j\}\subseteq J} 
    g_{\underline{n}\setminus\{i,j\}\sqcup\{\star\}}\circ_\star 
            g_{\{i,j\}} &= g_{I\sqcup\{\star\}}
                \circ_\star g_{J}\quad \text{ for $\underline{n}=I\sqcup J$ with $|J|>2$, $k=0,\ldots,d-1$}\\
\sum_{\{i,j\}\subset\underline{n}} 
    g_{\underline{n}\setminus\{i,j\}\sqcup\{\star\}}\circ_\star 
            g_{\{i,j\}} &=0 \quad \text{ for $n\geqslant 3$, $k=0,\ldots,d-1$},\\
 c^d=0, \quad &\text{ and }\quad  c\circ_1 g_I = g_I\circ_i c\quad  \text{ for all } i\in I.
 \end{align*}
Westerland established in \cite[Th.~1.2]{MR3068956} that the operad $H_{\bullet-1}(\CGK_d)$ is isomorphic to the suboperad of $\Grav_d$ generated by all elements of arity greater than one. An important ingredient of his argument is the computation of the cohomology algebra of $\CGK_d(n)$. We shall use that computation to establish the following auxiliary statement.

\begin{proposition}\label{prop:split}
The mixed Hodge structure on $H^\bullet(\CGK_d(n))$ is split. 
\end{proposition}

\begin{proof}
Recall that the mixed Hodge structure of a variety $X$ is said to be pure of weight $\frac{p}{q}$ with $\gcd(p,q)=1$, if it is concentrated in degrees divisible by $q$, and in degree $qa$ it is concentrated in the weight $pa$. Let us explain that, for any finite set of points $a_1,\ldots, a_k\in \mathbb{C}^d$, the mixed Hodge structure of 
 \[
 H^\bullet(\mathrm{Conf}(n,\mathbb{C}^d\setminus \{a_1,\ldots,a_k\}))    
 \]
is pure of weight $\frac{2d}{2d-1}$. (This is certainly known to the experts, for instance it can be inferred from the results of \cite{huang2024cohomologyconfigurationspacespunctured}, but we record it for completeness.) The argument goes as follows. Slightly modifying \cite[Def.~8.2]{MR4157110}, we call a finite set $\{H_i\}_{i\in I}$ of affine subspaces in a complex affine space a good arrangement of codimension $d$, if they are all of dimension $d$, and if for each $i\in I$, the nonempty subspaces in $\{H_i\cap H_j\}_{j\ne i}$ form a good arrangement of codimension $d$ in $H_i$. The proof of \cite[Prop.~8.6]{MR4157110} works \emph{mutatis mutandis} to establish that for any good arrangement $\{H_i\}_{i\in I}$ of codimension $d$ in a complex affine space~$V$, the mixed Hodge structure of $H^\bullet(V\setminus \cup_i H_i)$ is pure of weight $\frac{2d}{2d-1}$. We now consider the arrangement of codimension $d$ subspaces in $\mathbb{C}^{nd}=(\mathbb{C}^d)^n$ consisting of the diagonal subspaces $\Delta_{i,j}=\{z_i=z_j\}_{1\le i\ne j\le n}$ and the ``punctures'' $P_{i,s}=\{z_i=a_s\}_{1\le i\le n,1\le s\le k}$. They manifestly form a good arrangement of codimension $d$, proving the stated result. 

We now recall the observation of Westerland \cite[Sec.~2]{MR3068956} that there is a Fadell--Neuwirth type fibration
 \[
\mathrm{Conf}(n,\mathbb{C}^d\setminus \{a_1,a_2\})\hookrightarrow \mathrm{Conf}(n,\mathbb{C}^d)/(\mathbb{C}^d\rtimes\mathbb{C}^\times)\twoheadrightarrow  \mathrm{Conf}(2,\mathbb{C}^d)/(\mathbb{C}^d\rtimes\mathbb{C}^\times),    
 \]
that $\mathrm{Conf}(2,\mathbb{C}^d)/(\mathbb{C}^d\rtimes\mathbb{C}^\times)\cong \mathbb{C}P^{d-1}$, and that the Serre spectral sequence of this fibration abuts on page two, and that by degree reasons there is a multiplicative isomorphism 
 \[
E_2^{\bullet,\bullet}= H^\bullet(\mathbb{C}P^{d-1})\otimes H^\bullet(\mathrm{Conf}(n,\mathbb{C}^d\setminus \{a_1,a_2\})) \cong H^\bullet(\mathrm{Conf}(n,\mathbb{C}^d)/(\mathbb{C}^d\rtimes\mathbb{C}^\times))=H^\bullet(\CGK_d(n)).  
 \] 
It is established in various contexts (see, for instance, \cite[Th.~3]{MR1404924} or \cite[Th.~6.5]{MR2393625}) that the Leray spectral sequence of a morphism between quasi-projective varieties is strictly compatible with the mixed Hodge structures. It follows that 
the associated graded for the weight filtration on $H^\bullet(\CGK_d(n))$ is 
 \[
H^\bullet(\mathbb{C}P^{d-1})\otimes H^\bullet(\mathrm{Conf}(n,\mathbb{C}^d\setminus \{a_1,a_2\})).   
 \] 
It remains to notice that the homological degrees where the first factor does not vanish are $0,-2,\ldots,-(2d-2)$, while for the second factor the corresponding homological degrees are divisible by $2d-1$. Therefore, for each homological degree $m$ there is at most one way to write it as $m=p+q$, where $H^p(\mathbb{C}P^{d-1})\ne 0$ and $H^q(\mathrm{Conf}(n,\mathbb{C}^d\setminus \{a_1,a_2\}))\ne 0$. It follows that the mixed Hodge structure of $H^m(\CGK_d(n))$ is pure of the weight given by the sum of the weights of the mixed Hodge structures of the corresponding $H^p(\mathbb{C}P^{d-1})$ and $H^q(\mathrm{Conf}(n,\mathbb{C}^d\setminus \{a_1,a_2\}))$.
\end{proof}

In the rest of this section, we shall use the work of Westerland to describe an explicit presentation of the operad $H_{\bullet-1}(\CGK_d)$ by generators and relations, to prove that this operad is Koszul, and to determine its Koszul dual. 

\subsection{The presentation by generators and relations}

The main result of this section is the following theorem.

\begin{theorem}\label{th:GravIso}
The operad $\widetilde\Grav_d$ is generated by totally symmetric operations $\gamma_I^a\in \widetilde\Grav_d(I)$, $0\le a\le d-1$, of homological degree $-1-2d(|I|-1)+2d-2a$, subject to the following relations 
\begin{align*}
\sum_{\{i,j\}\subseteq J} 
    \gamma^k_{\underline{n}\setminus\{i,j\}\sqcup\{\star\}}\circ_\star 
            \gamma^{0}_{\{i,j\}} &= \gamma^k_{I\sqcup\{\star\}}
                \circ_\star \gamma^{0}_{J}\quad \text{ for\ $\underline{n}=I\sqcup J$ with $|J|>2$, $0\le k \le d-1$}\\
\sum_{\{i,j\}\subset\underline{n}} 
    \gamma^k_{\underline{n}\setminus\{i,j\}\sqcup\{\star\}}\circ_\star 
            \gamma^{0}_{\{i,j\}} &=0 \quad \text{ for $n\geqslant 3$, $0\le k \le d-1$},\\
   \gamma_{I\sqcup\{\star\}}^a\circ_\star \gamma^{a'}_J&=\gamma^{a+a'}_{I\sqcup\{\star\}}\circ_\star \gamma^{0}_J\quad \text{ for }\ 0\le a+a'\le d-1,\\
   \gamma_{I\sqcup\{\star\}}^a\circ_\star \gamma^{a'}_J&=0\quad \text{ for } a+a'\ge d.   
\end{align*}
\end{theorem}

\begin{proof}
As mentioned above, according to \cite[Th.~1.2]{MR3068956}, the operad $\{H_{\bullet-1}(\CGK_d(n)(\mathbb{C}))\}_{n\ge 1}$ is isomorphic to the suboperad of $\Grav_d$ generated by all elements of arity greater than one. Note that since the relations of $\Grav_d$ in particular imply that all operations $g_I\in \Grav_d(I)$ are $c$-linear, we have $\Grav_d\cong\Grav_d^0[c]/(c^d)$, where $\Grav_d^0\subset\Grav_d$ is the suboperad generated by all the operations $g_I$. In particular, this immediately implies that the suboperad of $\Grav_d$ generated by all elements of arity greater than one is generated by the operations $g_Ic^a$, where $0\le a\le d-1$. We have
 \[
 |\mathcal S^{2d}g_Ic^a|  =-2d(|I|-1)+2d-1-2a, 
 \]
which is precisely the homological degree of the generator $\gamma_I^a$ of the operad $\widetilde\Grav_d$. Moreover, a direct inspection shows that if we substitute $\mathcal S^{2d}g_Ic^{d-1-a}$ instead of $\gamma_I^a$ into the defining relations of the operad $\widetilde\Grav_d$, those relations are satisfied. Thus, there exists a surjective map of operads
 \[
\widetilde\Grav_d\twoheadrightarrow \mathcal S^{2d}\{H_{\bullet-1}(\CGK_d(n)(\mathbb{C}))\}_{n\ge 1}.
 \]

To proceed, we use a Gröbner basis argument. Let us define an order of shuffle tree monomials on the free operad generated by all $\gamma_I^a$ as a superposition of several partial orders. This is inspired by the proof of \cite[Prop.~3]{MR4392457}, where a Gröbner basis for this operad in the case $d=1$ is given. 

We first define a new weight function which is equal to one on all generators except for $\gamma_I^{0}$ on which it is equal to zero, extend it additively to compositions of generators, and then impose a partial order that compares monomials using the total weight. The only relations of $\calG^{(d)}$ that are not homogeneous with respect to this weight are the relations of the third group, where the leading monomials are the monomials are those on the left hand side, unless $a=0$. 

We then define another weight function that assigns weight $1$ to each generator $\gamma_I^a$ with $|I|>2$, and weight $0$ to each generator $\gamma_I^a$ with $|I|=2$, extend it additively to compositions of generators, and then impose a partial order that compares monomials using the total weight. This ensures that for $2<|J|<n-1$, the monomial on the right hand side of the corresponding relation of the first group of relations of $\calG^{(d)}$ is the leading monomial of that relation. 

Finally, let us consider the monoid of ``quantum monomials'' $\mathsf{QM}=\langle x,y,q\mid xq=qx,yq=qy,yx=xyq\rangle$; it is proved in \cite[Lemma 2.2]{MR4114993} that the binary relation $\prec$ such that 
 \[
x^ky^lq^m\prec x^{k'}y^{l'}q^{m'}\text{ if } k>k' \text{ or } k=k', l<l' \text{ or } k=k', l=l', m<m'     
 \]
makes $\mathsf{QM}$ an ordered monoid (the product is compatible with the order). We shall use the map from the free operad generated by all $\gamma_I^a$ to the word operad associated to $\mathsf{QM}$ sending $\gamma_I^a$ to $(y,y)$ if $|I|=2$ and $\gamma_I^a$ to $(x,x,\ldots,x)$ if $|I|>2$.  According to \cite[Prop.~1.6]{MR4114993}, this makes the shuffle word operad associated to $\mathsf{QM}$ a partially ordered operad, and we may pull back that partial order to a partial order of monomials in the free operad generated by all $\gamma_I^a$. This ensures that for $|S|=n-1$, the monomial on the right hand side of the corresponding relation of the first group of relations of $\calG^{(d)}$ is the leading monomial of that relation. 

The only relations for which our partial orderings do not yet give a definite answer for the leading monomial are the relations of the second group and the relations
 \[
\gamma_{I\sqcup\{\star\}}^k\circ_\star \gamma^{0}_J=\gamma^{0}_{I\sqcup\{\star\}}\circ_\star \gamma^k_J     
 \]
of the third group. For relations of the second group, it is known \cite[Prop.~3]{MR4392457} that it is convenient to choose the order for which the leading term of that relation is
 \[
\gamma^{k}_{\{1,\ldots,n-2,\star\}}\circ_\star\gamma^0_{\{n-1,n\}},    
 \]
which can be accomplished by taking, for instance, the reverse path-permutation lexicographic ordering \cite[Def.~5.4.1.8]{BD}. This will also break a tie between $\gamma_{I\sqcup\{\star\}}^k\circ_\star \gamma^{0}_J$ and $\gamma^{0}_{I\sqcup\{\star\}}\circ_\star \gamma^k_J$ if, in the definition of that ordering, we postulate that $\gamma^a_J>\gamma^b_J$ for $a<b$; for this choice, the leading monomial of the corresponding relation is $\gamma^{0}_{I\sqcup\{\star\}}\circ_\star \gamma^k_J$. 

As we mentioned above, for $d=1$ the relations are known to form a Gröbner basis \cite[Prop.~3]{MR4392457}. We note that the relations of the first and the second group are of the same form as those for $d=1$, and all computations of the S-polynomials between them that one needs to compute when computing the Gröbner basis mimic the computations of S-polynomials in the case $d=1$ (since all terms, including the leading term, have $\gamma^k_I$ with $k>0$ only at the root vertex, all terms of all S-polynomials will satisfy that property, and all corollas are of odd homological degrees, so the signs will be the same as for the gravity operad), so all those S-polynomials reduce to zero. The S-polynomials between elements of the third group and the elements of the first and second group can be seen to reduce to zero directly. Finally, the S-polynomials involving the monomial relations of the fourth group are manifestly reducible to zero. This implies that the defining relations of the shuffle operad associated to $\widetilde\Grav_d$ form a Gröbner basis; in particular, monomials that are normal with respect to those defining relations form a linear basis of that operad. These normal monomials are analogous to the known normal monomials for the gravity operad \cite[Th.~5.16]{MR3084563}, namely, for each arity $n>1$, all normal monomials are of the form
 \[
\gamma_J^a(\lambda_1,\ldots,\lambda_m)\in \widetilde\Grav_d(n),     
 \]
where $\lambda_i$ has inputs in the finite set $I_i$ which all form a partition $\underline{n}=I_1\sqcup\cdots\sqcup I_m$, each element $\lambda_i$ is a shuffle left comb with all vertices labelled $\gamma_{\underline{2}}^{0}$, and moreover $|I_m|=1$. In particular, the operad $\Grav_d^0$ is isomorphic to the suboperad of $\widetilde\Grav_d$ generated by the elements $\gamma_J^{0}$, and hence the number of normal monomials in $\widetilde\Grav_d(n)$ is $d$ times the number of normal monomials in $\Grav_d^0(n)$. But we know that $\Grav_d\cong\Grav_d^0[c]/(c^d)$, so $d\dim \widetilde{\Grav_d}(n)=\dim\Grav_d(n)$, and thus we just proved that for all $n>1$
 \[
\dim\widetilde\Grav_d(n)=\dim\Grav_d(n)=\dim \mathcal S^{2d}H_{\bullet-1}(\CGK_d(n)).
 \]
Thus, our previously constructed surjective map of operads is an isomorphism, which proves the theorem. 
\end{proof}

\begin{corollary}\label{cor:Koszul}
The operad $\widetilde\Grav_d$ is Koszul.
\end{corollary}

\begin{proof}
We saw in the proof of Theorem \ref{th:GravIso} that the shuffle operad associated to $\widetilde\Grav_d$ has a quadratic Gröbner basis of relations, and is therefore Koszul \cite{MR2667136}, which implies that the symmetric operad $\widetilde\Grav_d$ is Koszul. 
\end{proof}

\subsection{The Koszul dual operad}
Let us describe the Koszul dual operad. For that, we give the following definition.

\begin{definition}
The operad $\widetilde\HyperCom_d$ is generated by totally symmetric operations $m_I^a\in \widetilde\HyperCom_d(I)$, $0\le a\le d-1$, of homological degree $2(d(|I|-1)-a-1)$, subject to the following relations:
 \[
\sum_{\substack{I\sqcup J =\underline{n} \\ i\in I,\, j,k\in J\\ a+b=p}} 
m_{I\sqcup\{\star\}}^a\circ_\star m_{J}^b=
\sum_{\substack{I\sqcup J = \underline{n} \\ k\in I,\, i,j\in J,\\ a+b=p}} 
m^a_{I\sqcup\{\star\}}\circ_\star m^b_{J}     
 \]
for all $n\ge 3$, all triples of distinct elements $i,j,k\in\underline{n}$, and all $p\in \{0,\ldots,d-1\}$. 
\end{definition}

\begin{proposition}\label{prop:Koszuldual}
We have an operad isomorphism
 \[
\widetilde\Grav_d\cong \mathcal S\widetilde\HyperCom_d^! .
 \]
In particular, the operad $\widetilde\HyperCom_d$ is Koszul. 
\end{proposition}

\begin{proof}
The Koszul dual operad $\widetilde\HyperCom_d^!$ is generated by the operadic desuspension of the duals of suspensions of generators of $\widetilde\HyperCom_d$, so the operadic suspension of the operad $\widetilde\HyperCom_d^!$ is generated by the suspensions of generators of $\widetilde\HyperCom_d$. The generator $m_I^a$ has the homological degree $2(d(|I|-1)-a-1)$, so the dual of its suspension has the homological degree $-2(d(|I|-1)-a-1)-1=2a+1-2d(|I|-1)$, which suggests that the duality should relate it to $\gamma_I^{d-1-a}$, which we shall do throughout the proof.

Furthermore, we note that the relations of $\widetilde\Grav_d$ annihilate the relations of $\widetilde\HyperCom_d$ under the pairing between the weight two components of the respective free operads (the desuspension in particular ensures that we must compute the usual pairing without sign twists). Moreover, examining the relations of $\widetilde\Grav_d$, we immediately see that for a basis of the vector space of elements of arity $n$ and weight two in the quotient operad, we may take all elements of the form $\gamma^k_{\underline{n}\setminus\{i,j\}\sqcup\{\star\}}\circ_\star \gamma^{0}_{\{i,j\}}$ with $\{i,j\}\ne \{n-1,n\}$ and $k=0,\ldots,d-1$. Let us explain that there is the exact same number of linearly independent relations of $\widetilde\HyperCom_d$. Indeed, if we add the element
 \[
 \sum_{\substack{I\sqcup J =\underline{n} \\ i,j,k\in J\\ a+b=p}} m_{I\sqcup\{\star\}}^a\circ_\star m_{J}^b    
 \]
to both sides of a relation of $\widetilde\HyperCom_d$ indexed by $i,j,k$, that relation becomes
 \[
\sum_{\substack{I\sqcup J =\underline{n} \\ j,k\in J\\ a+b=p}} 
m_{I\sqcup\{\star\}}^a\circ_\star m_{J}^b=
\sum_{\substack{I\sqcup J = \underline{n} \\ i,j\in J,\\ a+b=p}} 
m^a_{I\sqcup\{\star\}}\circ_\star m^b_{J} ,    
 \]
showing that the element
 \[
\sum_{\substack{I\sqcup J =\underline{n} \\ i,j\in J\\ a+b=p}} 
m_{I\sqcup\{\star\}}^a\circ_\star m_{J}^b     
 \]
does not depend on $i$. It then easily follows that all these relations are equivalent to the linearly independent
relations
 \[
\sum_{\substack{I\sqcup J =\underline{n} \\ i,j\in J\\ a+b=p}} 
m_{I\sqcup\{\star\}}^a\circ_\star m_{J}^b=
\sum_{\substack{I\sqcup J = \underline{n} \\ n-1,n\in J,\\ a+b=p}} 
m^a_{I\sqcup\{\star\}}\circ_\star m^b_{J} ,    
 \]
and we have these relations for $\{i,j\}\ne \{n-1,n\}$ and $0\le p\le d-1$, which is precisely the data indexing the basis indicated above. This completes the proof of the Koszul duality between the two operads. Since the operadic (de)suspension and passing to the Koszul dual do not affect the Koszul property, Corollary \ref{cor:Koszul} implies that the operad $\widetilde\HyperCom_d$ is Koszul.
\end{proof}

Results of Getzler \cite{MR1363058} imply that the operad $\widetilde\HyperCom_1$, known as $\HyperCom$, is isomorphic to the homology operad of the operad of Deligne--Mumford spaces $\overline{\calM}_{0,n+1}$. In what follows, we shall see that for $d>1$ the homology operad of the operad of Chen--Gibney--Krashen spaces is a nontrivial deformation of the operad $\widetilde\HyperCom_d$.

\section{The hypercommutative operad and its properties}\label{sec:CGK}

In this section, we shall discuss various properties of the operad 
 \[
\HyperCom_d:= H_\bullet(\bCGK_d),    
 \]
which we call the $d$-dimensional hypercommutative operad. We begin with giving a presentation of this operad by generators and relations.

\begin{theorem}\label{th:CGK}
The operad $\HyperCom_d$ is generated by totally symmetric operations 
 \[
\mu^{(p)}_n\in \HyperCom_d(n),\quad p\ge 0,    
 \]
of homological degree $2(d(n-1)-1-p)$, whose defining relations declare that for each $n$, for each $i\ne j\in\underline{n}$, and for each $m\ge d$, we have
 \[
\sum_{q\ge1}\sum_{\substack{S_1\sqcup \cdots\sqcup S_q=\underline{n},\\S_1,\ldots,S_q\ne\varnothing\colon\{i,j\}\subseteq S_1
}}
\sum_{p_1+\cdots+p_q+q-1=m}
\mu^{(p_q)}_{S_q\sqcup\{\star\}}\circ_\star\cdots\circ_\star\mu^{(p_2)}_{S_2\sqcup\{\star\}}\circ_\star\mu^{(p_1)}_{S_1}=0.
 \]
\end{theorem}

\begin{proof}
Let us define, for a given $n\ge 2$, for a subset $\{i,j\}\subseteq\underline{n}$, and for an integer $k\ge 0$, the following elements in the operad $\HyperCom_d$:
\begin{gather*}
\mu_n^{(m)}:=\mathrm{PD}(\sigma_{\{1,\ldots,n\}}^m),\\ 
\nu_{n,i,j}^{(m)}:=\mathrm{PD}(\sigma_{\{i,j\}}^m). 
\end{gather*}
Note that we have 
\begin{multline*}
\nu_{n,i,j}^{(m)}=\mathrm{PD}(\sigma_{\{i,j\}}^m)\\=\mathrm{PD}(\sigma_{\{i,j\}}^m-\sigma_{\{1,\ldots,n\}}^m+\sigma_{\{1,\ldots,n\}}^m)
=\mathrm{PD}\left(\sum_{\{i,j\}\subseteq J\subsetneq\underline{n}}x_J\sum_{a=1}^{m}\sigma_{\{i,j\}}^{a-1}\sigma_{\{1,\ldots,n\}}^{m-a}\right)+
\mathrm{PD}(x_{\{1,\ldots,n\}}^m)\\=\sum_{\{i,j\}\subseteq J\subsetneq\underline{n}}\mathrm{PD}\left(x_J\sum_{a=1}^{m}\sigma_{\{i,j\}}^{a-1}\sigma_{\{1,\ldots,n\}}^{m-a}\right)+\mu_n^{(m)}
=\sum_{\{i,j\}\subseteq J\subsetneq\underline{n}}\sum_{a=1}^{m}\mu_{J^c\sqcup\{\star\}}^{(m-a)}\circ_\star\nu_{J,i,j}^{(a-1)}+\mu_n^{(m)}.
\end{multline*}
Here we used the fact that $\sigma_{\underline{n}}=x_{\underline{n}}$, so that 
 \[
\sigma_{\{i,j\}}-\sigma_{\{1,\ldots,n\}}=\sum_{\{i,j\}\subseteq J\subsetneq\underline{n}}x_J,     
 \]
as well as the description of the operad structure in terms of cohomology rings from Proposition~\ref{prop:operad-on-cohomology}. The elements $\nu_{J,i,j}^{(a-1)}$ in this formula can be expanded using the same scheme of computation, and, since we clearly have $\mu_{\{i,j\}}^{(m)}=\nu_{\{i,j\},i,j}^{(m)}$, an easy inductive argument gives us the formula 
 \[
\nu_{n,i,j}^{(m)}=
\sum_{q\ge1}\sum_{\substack{S_1\sqcup \cdots\sqcup S_q=\underline{n},\\S_1,\ldots,S_q\ne\varnothing\colon\{i,j\}\subseteq S_1
}}
\sum_{p_1+\cdots+p_q=m-q+1}
\mu^{(p_q)}_{S_q\sqcup\{\star\}}\circ_\star\cdots\circ_\star\mu^{(p_2)}_{S_2\sqcup\{\star\}}\circ_\star\mu^{(p_1)}_{S_1}.
 \]
Recalling that $\sigma_{\{i,j\}}^m=0$ for $m\ge d$, we see that the relations listed in the statement of the theorem do indeed hold in the operad $\HyperCom_d$. It remains to show that there are no other relations. Let us make several observations. First of all, writing the relation above as 
 \[
\sum_{q\ge2}\sum_{\substack{S_1\sqcup \cdots\sqcup S_q=\underline{n},\\S_1,\ldots,S_q\ne\varnothing\colon\{i,j\}\subseteq S_1
}}
\sum_{p_1+\cdots+p_q=m-q+1}
\mu^{(p_q)}_{S_q\sqcup\{\star\}}\circ_\star\cdots\circ_\star\mu^{(p_2)}_{S_2\sqcup\{\star\}}\circ_\star\mu^{(p_1)}_{S_1}=-\mu^{(m)}_n, 
 \]
we see that it is possible to express each operation $\mu^{(m)}_n$ with $m\ge d$ via operadic compositions of operations of lower arities. Thus, the operad $\HyperCom_d$ is generated by the operations $\mu^{(p)}_n$ with $0\le p\le d-1$. Furthermore, using the fact that the operation $\mu^{(m)}_n$ is totally symmetric, we see that 
  \[
\sum_{q\ge2}\sum_{\substack{S_1\sqcup \cdots\sqcup S_q=\underline{n},\\S_1,\ldots,S_q\ne\varnothing\colon\{i,j\}\subseteq S_1
}}
\sum_{p_1+\cdots+p_q=m-q+1}
\mu^{(p_q)}_{S_q\sqcup\{\star\}}\circ_\star\cdots\circ_\star\mu^{(p_2)}_{S_2\sqcup\{\star\}}\circ_\star\mu^{(p_1)}_{S_1}
 \]
does not depend on $i,j$. Rewriting this relation in terms of the minimal set of generators $\mu^{(p)}_n$ with $0\le p\le d-1$ does not seem to lead to elegant formulas, but we can keep track of the quadratic component of this expression (that is, the combination of terms involving two operations $\mu_n^{(p)}$ from the minimal set of generators). Indeed, rewriting redundant generators will turn quadratic terms into terms that are compositions of at least three operations, so we may safely ignore those. Thus, the quadratic component of this expression is 
  \[
\sum_{\substack{S_1\sqcup S_2=\underline{n},\\S_1,S_2\ne\varnothing\colon\{i,j\}\subseteq S_1
}}
\sum_{0\le p_1,p_2\le d-1, p_1+p_2=m-1}
\mu^{(p_2)}_{S_2\sqcup\{\star\}}\circ_\star\mu^{(p_1)}_{S_1}, 
 \]
so for $d\le m\le 2d-1$, the assertion that this does not depend on $i,j$ is equivalent to the defining relations of the operad $\widetilde\HyperCom_d$, see the proof of Proposition \ref{prop:Koszuldual}, and for $m\ge 2d$ every quadratic element will be a composition of two generators at least one of which is redundant, so the quadratic parts of the relations are trivial. Thus, if we define a decreasing filtration of our operad by the number of generators involved in a composition, in the associated graded operad the cosets of the operations $\mu_n^{(p)}$ with $0\le p\le d-1$ satisfy the relations of the operad $\widetilde\HyperCom_d$. To show that the operad $\HyperCom_d$ has no other relations, it is enough to show that the Betti numbers of the space $\bCGK_d(n)(\mathbb{C})$ are equal to the dimensions of graded components of $\widetilde\HyperCom_d$. 

Our argument relies on the following observation of Petersen \cite[Ex.~3.9]{MR3654116} giving a general framework to an argument of Getzler and Jones \cite[Sec.~3.3]{getzler1994operadshomotopyalgebraiterated}. Suppose that we have a topological operad $\{X(n)\}$ stratified in a way that the strata of $X(n)$ correspond to rooted trees with $n$ leaves, the closed stratum $X(\tau)$ corresponding to a tree $\tau$ satisfies
 \[
X(\tau)\cong \prod_{v\in V(\tau)}X(\mathrm{in}(v)),
 \]
and the composition maps in the operad are given tautologically by inclusions of closed strata. Let us denote by $Y(n)$ the open stratum corresponding to the $n$-corolla (the unique tree with one vertex). We have a cooperad structure on $\{H^\bullet_c(X(n))\}$ coming from the operad structure on $\{X(n)\}$ and an operad structure on $\{H^{\bullet+1}_c(Y(n))\}$ coming from the fact that the open stratum $Y(I\sqcup\{\star\})\times Y(J)$ is adjacent to $Y(I\sqcup J)$, and one can consider the connecting homomorphism in the long exact sequence of a pair
 \[
H^{\bullet}_c(Y(I\sqcup\{\star\})\times Y(J))) \to H^{\bullet-1}_c(Y(I\sqcup J)).   
 \] 
According to \cite[Ex.~3.9]{MR3654116}, the spectral sequence of the filtration on $X(n)$ by the number of vertices of a tree has as its page $E_1^{(1)}$ the bar construction of the operad $\{H^{\bullet+1}_c(Y(n))\}$ and as its page $E_\infty^{(1)}$ the associated graded for a filtration on $\{H^\bullet_c(X(n))\}$. We apply this general observation in the case where $X(n)=\bCGK_d(n)$ and $Y(n)=\CGK_d(n)$. As we indicated above, the operad structure on $\{H^{\bullet+1}_c(\CGK_d(n))\}$ is identified under the Poincaré duality isomorphism with the operad $\mathcal S^{2d}\{H_{\bullet-1}(\CGK_d(n))\}=\mathcal S^{2d}\widetilde\Grav_d$.
According to Proposition \ref{prop:Koszuldual} and the usual formulas for Koszul dual (co)operads \cite[Sec.~7.2.3]{LV}, we have 
 \[
\mathcal S^{2d}\widetilde\Grav_d\cong \mathcal S\widetilde\HyperCom_d^!=(\widetilde\HyperCom_d^{\ac})^*.  
 \]
Moreover, according to Corollary \ref{cor:Koszul} and Proposition \ref{prop:Koszuldual}, the operads $\widetilde\Grav_d$ and $\widetilde\HyperCom_d$ are Koszul, and therefore the homology of the bar construction of the operad $(\widetilde\HyperCom_d^{\ac})^*$ is isomorphic to the cooperad $\widetilde\HyperCom_d^*$. The underlying graded vector space of that cooperad is supported at even homological degrees, so there is no room for further differentials in the spectral sequence $E^{(1)}$. Thus, it abuts on the second page, and therefore the associated graded for a filtration on $H^{\bullet}_c(\bCGK_d)$ is isomorphic to the cooperad $\widetilde\HyperCom_d^*$, which gives the desired coincidence of the Betti numbers of $\bCGK_d(n)$ with the dimensions of graded components of $\widetilde\HyperCom_d(n)$. We conclude that no other relations are satisfied in the operad $\HyperCom_d$, as required. 

\end{proof}

As an example, let us consider the case $d=2$. We saw in the proof that the minimal set of generators of the operad $\HyperCom_2$ consists of the generators $\mu_n^{(0)}$ and $\mu_n^{(1)}$ for all $n\ge 2$. Let us demonstrate how some of the first relations between these generators are obtained. In arity $3$, among the relations of Theorem \ref{th:CGK}, we have 
\begin{gather*}
\mu_3^{(2)}+\mu_2^{(1)}\circ_1\mu_2^{(0)}+ \mu_2^{(0)}\circ_1\mu_2^{(1)}=0,\\
\mu_3^{(2)}+\mu_2^{(1)}\circ_2\mu_2^{(0)}+ \mu_2^{(0)}\circ_2\mu_2^{(1)}=0,\\
\mu_3^{(3)}+\mu_2^{(2)}\circ_1\mu_2^{(0)}+ \mu_2^{(1)}\circ_1\mu_2^{(1)}+ \mu_2^{(0)}\circ_1\mu_2^{(2)}=0,\\    
\mu_3^{(3)}+\mu_2^{(2)}\circ_2\mu_2^{(0)}+ \mu_2^{(1)}\circ_2\mu_2^{(1)}+ \mu_2^{(0)}\circ_2\mu_2^{(2)}=0.    
\end{gather*}
Note that $\mu_2^{(2)}=0$, so these relations imply
\begin{gather*}
\mu_2^{(1)}\circ_1\mu_2^{(0)}+ \mu_2^{(0)}\circ_1\mu_2^{(1)}=\mu_2^{(1)}\circ_2\mu_2^{(0)}+ \mu_2^{(0)}\circ_2\mu_2^{(1)},\\
\mu_2^{(1)}\circ_1\mu_2^{(1)}= \mu_2^{(1)}\circ_2\mu_2^{(1)}.    
\end{gather*}
In arity $4$, among the relations of Theorem \ref{th:CGK}, we have (partially simplifying using the relations $\mu_2^{(2)}=0$ and $\mu_3^{(2)}+\mu_2^{(1)}\circ_2\mu_2^{(0)}+ \mu_2^{(0)}\circ_2\mu_2^{(1)}=0$)
\begin{gather*}
\mu_4^{(2)}+\mu_3^{(1)}\circ_1\mu_2^{(0)}+\mu_3^{(0)}\circ_1\mu_2^{(1)}+\bigl(\mu_2^{(1)}\circ_1\mu_3^{(0)}+\mu_2^{(0)}\circ_1\mu_3^{(1)}+
  \mu_2^{(0)}\circ_1\mu_2^{(0)}\circ_1\mu_2^{(0)}\bigr).(1+(3,4))=0,\\     
\mu_4^{(3)}+\mu_3^{(2)}\circ_1\mu_2^{(0)}+\mu_3^{(1)}\circ_1\mu_2^{(1)}+\bigl(\mu_2^{(1)}\circ_1\mu_3^{(1)}+
  \mu_2^{(1)}\circ_1\mu_2^{(0)}\circ_1\mu_2^{(0)}\bigr).(1+(3,4))=0,
\end{gather*}
and simplifying further gives us the relations
\begin{multline*}
\mu_3^{(1)}\circ_1\mu_2^{(0)}+\mu_3^{(0)}\circ_1\mu_2^{(1)}+(\mu_2^{(1)}\circ_1\mu_3^{(0)}+\mu_2^{(0)}\circ_1\mu_3^{(1)}).(3,4)+\bigl(
  \mu_2^{(0)}\circ_1\mu_2^{(0)}\circ_1\mu_2^{(0)}\bigr).(1+(3,4))=\\     
\mu_3^{(1)}\circ_2\mu_2^{(0)}+\mu_3^{(0)}\circ_2\mu_2^{(1)}+(\mu_2^{(1)}\circ_2\mu_3^{(0)}+\mu_2^{(0)}\circ_2\mu_3^{(1)}).(1,4)+\bigl(
  (\mu_2^{(0)}\circ_1\mu_2^{(0)})\circ_2\mu_2^{(0)}\bigr).(1+(1,4))
\end{multline*}
and 
\begin{multline*}
\mu_3^{(1)}\circ_1\mu_2^{(1)}+(\mu_2^{(1)}\circ_1\mu_3^{(1)}+\mu_2^{(1)}\circ_1\mu_2^{(0)}\circ_1\mu_2^{(0)}).(3,4)-\mu_2^{(0)}\circ_1\mu_2^{(1)}\circ_1\mu_2^{(0)}=\\
\mu_3^{(2)}\circ_2\mu_2^{(0)}+\mu_3^{(1)}\circ_2\mu_2^{(1)}+(\mu_2^{(1)}\circ_1\mu_3^{(1)}+\mu_2^{(1)}\circ_1\mu_2^{(0)}\circ_1\mu_2^{(0)}).(1,4)-(\mu_2^{(0)}\circ_1\mu_2^{(1)})\circ_2\mu_2^{(0)}
\end{multline*}
At this point it is natural to note that, while there is no freedom in choosing the generator $\mu_n^{(0)}$, since the corresponding component of the operad $\HyperCom_d$ is one-dimensional, there are some choices for the generator $\mu_n^{(p)}$ with $p>0$. For example, in the case $d=2$ that we are considering, the generator $\mu_3^{(1)}$ has degree $4$, and so in principle we can add to it a scalar multiple of the combination 
 \[
(\mu_2^{(0)}\circ_1\mu_2^{(0)})(1+(1,2,3)+(1,3,2)),   
 \]
which is the only decomposable element of the same arity and degree that is totally symmetric in its arguments. A simple direct verification shows that such a modification does not lead to any major simplification of the relations of the operad $\HyperCom_2$. 

We now continue to exhibit a certain type of Koszul duality between the operads $\HyperCom_d$ and $\widetilde\Grav_d$. Since we just established that the operad $\HyperCom_d$ is not even quadratic for $d>1$, these operads cannot be Koszul dual literally. However, quadratic parts of the relations of the operad $\HyperCom_d$ are precisely the relations of the operad $\widetilde\HyperCom_d$ which is Koszul and has $\widetilde\Grav_d^*$ as its Koszul dual cooperad, which suggests to look for the correct result in the realm of homotopy Koszul operads \cite[Sec.~5.4]{MR2560406}. This is exactly what we shall do. Precisely, we recall that there is a notion of operads up to homotopy that goes back to van der Laan \cite{vdLaan}, see \cite[Sec.~4]{MR2560406} and \cite[Sec.~10.5]{LV} for more details. One can think about this notion using groupoid coloured operads \cite{MR3134040,MR4425832} as follows. One may encode operads using an operad $\mathsf{Op}$ coloured by the groupoid $\mathbb{S}$ of finite sets and bijections; this operad is Koszul, and operads up to homotopy are algebras over the coloured operad $\mathsf{Op}_\infty$ which is the cobar construction of the Koszul dual cooperad of $\mathsf{Op}$. Similarly to how in associative algebras the only basic structure operation is the binary product and in $A_\infty$-algebras there are operations with any number of arguments, basic operation of an operad are infinitesimal compositions along 2-vertex trees, while an $\mathsf{Op}_\infty$-structure involves higher operations corresponding to any possible rooted tree. 

A consequence of the Koszul property of the operad $\mathsf{Op}$ is that an $\mathsf{Op}_\infty$-structure can be transferred from a chain complex to its homology using the general framework of the homotopy transfer theorem \cite[Sec.~10.3]{LV}. As always, an $\mathsf{Op}_\infty$-structure on a chain complex with zero differential is said to be minimal. To state the following result, we also need to recall that, similarly to how an $A_\infty$-coalgebra on a vector space $V$ is the same as a differential on $\overline{T}(s^{-1}V)$, making this free algebra into a dg algebra, the cobar construction $\Omega_\infty(V)$, an $\mathsf{Op}_\infty$-coalgebra on $\calV$ is the same as a differential on $\overline{\calT}(s^{-1}\calV)$, making this free operad into a dg operad, the cobar construction $\Omega_\infty(\calV)$.

\begin{theorem}\label{th:MHShigher}
The operad structure on 
 \[
\widetilde\Grav_d\cong \mathcal{S}^{2d}H_{\bullet-1}(\CGK_d)     
 \]
can be extended to a minimal $\mathsf{Op}_\infty$-structure for which we have a quasi-isomorphism
 \[
\Omega_\infty(\widetilde\Grav_d^*)\sim \HyperCom_d.    
 \] 
\end{theorem}

\begin{proof}
We continue using the observation of \cite[Ex.~3.9]{MR3654116} for the spaces $X(n)=\bCGK_d(n)$ and $Y(n)=\CGK_d(n)$, as in the  proof of Theorem \ref{th:CGK}. According to \cite[Ex.~3.9]{MR3654116}, the spectral sequence $E^{(2)}$ of \cite[Th.~1]{MR3654116} has in this case as its page $E_1^{(2)}$ the cobar construction of the cooperad $\{H^\bullet_c(X(n),\mathbb{Q})\}$, and as its page $E_\infty^{(2)}$ the associated graded for a filtration on $\{H^{\bullet+1}_c(Y(n),\mathbb{Q})\}$. 

In our case, $E_1^{(2)}$ is the cobar construction of the cooperad $\{H^\bullet_c(X(n))\}\cong\HyperCom_d^*$. Moreover, since the complement of $Y(n)=\CGK_d(n)$ in $X(n)=\bCGK_d(n)$ is a normal crossing divisor, the stratification is by intersections of its components. As observed by Petersen \cite[Ex.~3.5]{MR3654116}, our spectral sequence becomes the Poincaré dual of the spectral sequence used by Deligne \cite[Sec.~3.2]{MR498551} to define a mixed Hodge structure on the cohomology of a smooth noncompact complex algebraic variety; it converges to the associated graded with respect to the weight filtration on the cohomology of that variety, in our case of $H^\bullet(\CGK_d(n))$. Note that by \cite[Cor.~3.2.13]{MR498551}, the Deligne spectral sequence abuts on the second page. Moreover, according to Proposition~\ref{prop:split}, the mixed Hodge structure of $H^\bullet(\CGK_d(n))$ is split, and hence the spectral sequence in fact converges to $H^\bullet(\CGK_d(n))$, and not merely to the associated graded.  

Since the Poincaré duality is compatible with mixed Hodge structures \cite[Sec.~6.3]{MR2393625}, all the differentials of our spectral sequence (the Poincaré dual of the Deligne spectral sequence) are strictly compatible with the mixed Hodge structures. Thus, we have isomorphisms
 \[
H_\bullet(\Omega(\HyperCom_d^*))\cong H^{\bullet+1}_c(\CGK_d)\cong\mathcal{S}^{2d}H_{\bullet-1}(\CGK_d)\cong\widetilde\Grav_d,     
 \] 
or, dually, 
 \[
H_\bullet(\mathsf{B}(\HyperCom_d))\cong\widetilde\Grav_d^*.     
 \] 
If we equip the cooperad $\widetilde\Grav_d^*$ with the minimal $\mathsf{coOp}_\infty$-structure obtained from the cooperad structure on the bar construction using the homotopy transfer \cite{MR2349311}, we have 
 \[
\Omega_\infty(\widetilde\Grav_d^*)\cong H_\bullet(\bCGK_d),   
 \] 
as required.  
\end{proof}

\begin{remark}
It is only at this point that the difference between the case $d=1$ and the case $d>1$ becomes crucial. Namely, in the case $d=1$ the mixed Hodge structure on $H^k(\CGK_1(n))$ is pure of weight $2k$ and hence there is no room for a nontrivial $\mathsf{Op}_\infty$-structure on the homology of the bar construction of the operad $H_\bullet(\bCGK_1)$, and therefore the operad $H_\bullet(\bCGK_1)$ is Koszul (this is behind the scenes of the argument of Getzler in \cite[Sec.~3.11]{MR1363058}). By contrast, this is no longer the case for $d>1$; one may argue that the $\mathsf{Op}_\infty$-structure on $\widetilde\Grav_d\cong\mathcal{S}^{2d}H_{\bullet-1}(\CGK_d)$ is, in a sense, a ``witness'' of the non-pure mixed Hodge structure on $H^\bullet(\CGK_d)$.  (A similar phenomenon was exhibited by McCrory \cite{MR757958} who described the weight filtration of the link of a normal surface singularity in terms of Massey products on the cohomology.) 
\end{remark}

To conclude this section, let us recall that in \cite{rossi2024chainlevelkoszulduality}, Rossi and Salvatore lifted the Koszul duality between the operads $\Grav$ and $\HyperCom$ to the chain level, and conjectured that the Deligne--Mumford operad is cellular. Both of these questions make perfect sense for any $d$: it would be interesting to determine if the homotopy Koszul duality between $\widetilde{\Grav}_d$ and $\HyperCom_d$ can be lifted to the chain level (where, in fact, the corresponding statement might be more transparent), and whether the Chen--Gibney--Krashen operad $\bCGK_d$ is cellular for any $d$. 

\section{The \texorpdfstring{$S^1$}{S1}-framed operad of little disks}\label{sec:S1framed}

We shall now discuss the operad of $\D_{2d}\rtimes S^1$ of $S^1$-framed little $2d$-disks, so that we can eventually elucidate its relationship with the Chen--Gibney--Krashen operad. The results of the previous sections, paired with the work of Ward \cite{MR3995021}, suggest this kind of relationship, and we shall see that it can be explained in rather direct way. 

We begin by giving a small $\mathbb{R}$-model of the operad $\D_{2d}\rtimes S^1$. Let us adopt the viewpoint of Khoroshkin and Willwacher \cite{khoroshkin2025realmodelsframedlittle} who constructed functorial models of framed little disk operads $\D_m\rtimes G$ that one can define for any subgroup $G\subset SO_m$. For even $m=2d$, this leads to a proof of formality of the operad $\D_{2d}\rtimes SO_{2d}$, which, however, does not immediately imply such formality result for any subgroup $G\subset SO_{2d}$. (The reason for it is that the inclusion of groups $G\hookrightarrow SO_{2d}$ need not be formal.) 

Let us give a short summary of the parts of \cite{khoroshkin2025realmodelsframedlittle} that is relevant to us, translating it into language suitable for our purposes. Let $\mathsf{Graphs}_{m}$ be the Kontsevich's graph operad. The underlying graded vector space of $\mathsf{Graphs}_{m}(n)$ has a basis of directed graphs with numbered white vertices bijectively numbered by $\underline{n}$ and unnumbered black vertices of degree $m$, and with edges of degree $1-m$, where changing the direction of an edge multiplies the graph by $(-1)^m$, where each black vertex is at least trivalent, and where no connected component may consist of black vertices only. The composition is defined by the ``usual'' formulas (insert a graph in a vertex and sum over all ways to reconnect the incoming edges of that vertex), and a differential comes from a suitable version of the operadic twisting construction \cite{MR4621635,MR3348138}. It is proved by Kontsevich \cite{MR1718044} that $\mathsf{Graphs}_m$ is a model of $\D_m$. It follows from \cite{khoroshkin2025realmodelsframedlittle} that the action of a Lie subgroup $G\subset SO_m$ on $\D_m$ can be encoded by a homotopy action of the abelian Lie algebra $\mathfrak{g}_G=\pi^{\mathbb{R}}(G)$ on the differential graded operad $\mathsf{Graphs}_m$. Specifically, one can always construct a model of $\D_m\rtimes G$ of the form
 \[
\mathsf{Graphs}_m\rtimes_{\alpha_m} U(\hat{\mathfrak{g}}_G),    
 \] 
where $\hat{\mathfrak{g}}_G=(\Lie(s^{-1}\overline{H}_\bullet(BG)),\partial)$ is the Lie bar-cobar resolution of $\mathfrak{g}_G$, and where $\alpha_m$ is a concrete Maurer--Cartan element of the dg Lie algebra $\overline{H}_\bullet(BG)\hat{\otimes}\mathsf{GC}_{m}^+$. Here $\mathsf{GC}_{m}^+$ is a certain Lie algebra of graphs (obtained as the one-dimensional extension of the Lie algebra $\mathsf{GC}_{m}$, which is defined as a suitable deformation complex \cite{MR4621635,MR3348138}). Let us apply this to the case where $m=2d$ is even and $G=S^1\cong SO_2$ is embedded in $SO_{2d}$ diagonally. In this case, 
 \[
\hat{\mathfrak{g}}_{S^1}=(\Lie(\Delta_k \colon k\ge 1),\partial)    
 \]
with $|\Delta_k|=2k-1$ and 
 \[
 \partial(\Delta_k)+\sum_{a+b=k}[\Delta_a,\Delta_b]=0.   
 \]
Moreover, in the case of even $m=2d$, the Maurer--Cartan element $\alpha_{2d}$ is much simpler than in general, and the argument along the lines of \cite[Sec.~5.2]{khoroshkin2025realmodelsframedlittle} shows that one obtains a model for $\D_{2d}\rtimes S^1$ given by $\Gerst_{2d}\rtimes U(\hat{\mathfrak{g}}_{S^1})$,    
where $\Gerst_{2d}=H_\bullet(\D_{2d})$ is the $2d$-Gerstenhaber operad, the semi-direct product is taken with respect to the action of $\hat{\mathfrak{g}}_{S^1}$ on $\Gerst_{2d}$ by derivations for which 
 \[
\Delta_k(-\cdot-)=
\begin{cases} 
\quad 0 , \qquad k\ne d,\\
[-,-] , \quad k=d.
\end{cases},
\qquad 
\Delta_k([-,-])=0.
 \]

It turns out that this model can be meaningfully simplified. Let us explain how that is done. 

\begin{definition}
For $d\ge 1$, we define the \emph{$d$-truncated model of the circle} $\hat{\mathfrak{g}}^{(d)}_{S^1}$ to be the differential graded Lie algebra with generators $\Delta_1$, \ldots, $\Delta_d$ of degrees $|\Delta_k|=2k-1$ satisfying 
\begin{gather}
\partial(\Delta_k)=-\sum_{a+b=k}[\Delta_a,\Delta_b],\\
\sum_{a+b=m}[\Delta_a,\Delta_b]=0, \quad d+1\le m\le 2d.
\end{gather}
\end{definition}

\begin{remark}
The above relations can be compactly written as 
 \[
[\partial+\Delta_1+\cdots+\Delta_k,\partial+\Delta_1+\cdots+\Delta_k]=0,    
 \]
where one adopts the usual ``differential trick'' convention $[\partial,a]=\partial(a)$. 
\end{remark}

We are now ready to state and prove an important technical result. 

\begin{proposition}\label{prop:d-model}
The operad $\D_{2d}\rtimes S^1$ has a model 
 \[
\Gerst_{2d}\rtimes U(\hat{\mathfrak{g}}^{(d)}_{S^1}),   
 \]
where the semi-direct product is taken with respect to the action of $\hat{\mathfrak{g}}^{(d)}_{S^1}$ on $\Gerst_{2d}$ by derivations for which 
 \[
\Delta_k(-\cdot-)=
\begin{cases} 
\quad 0 , \qquad k< d,\\
[-,-] , \quad k=d.
\end{cases},
\qquad 
\Delta_k([-,-])=0.
 \]
\end{proposition}

\begin{proof}
The action of $\hat{\mathfrak{g}}_{S^1}$ on $\Gerst_{2d}$ clearly factors through $\hat{\mathfrak{g}}^{(d)}_{S^1}$, and the semi-direct product construction preserves weak equivalences. Thus, we just need to prove that the obvious homomorphism 
$U(\hat{\mathfrak{g}}_{S^1})\to U(\hat{\mathfrak{g}}^{(d)}_{S^1})$ sending all generators $\Delta_k$ for $k>d$ to zero is a quasi-isomorphism. For that, it is enough to prove that the map $\phi_d\colon U(\hat{\mathfrak{g}}^{(d)}_{S^1})\to H_\bullet(S^1)$ given by 
 \[
\phi(\Delta_k)=
\begin{cases}
[S^1], \quad k=1,\\
\,\,\, 0, \qquad k>1
\end{cases}
 \]
is a quasi-isomorphism. We first note that, since the homology of the circle is the Koszul dual algebra of the polynomial algebra $A=\k[t]$, we have a quasi-isomorphism
 \[
\Omega(A^\vee)\sim H_\bullet(S^1),  
 \]
where $A^\vee$ is the graded dual coalgebra of $A$. Let us now consider the following algebra $\tilde{A}$: it has generators $c_1$, \ldots, $c_d$, and relations 
\begin{gather}
c_ic_j=c_{i+j}, \quad i+j\le d,\\ 
c_ic_j=c_kc_l, \quad d+1\le i+j=k+l\le 2d.
\end{gather}
Note that there is a surjective homomorphism of algebras $\tilde{A}\to A$ sending $c_k$ to $t^k$. We observe that the relations of the algebra $\tilde{A}$ form a Gröbner basis of the ideal they generate (for the graded lexicographic ordering with $c_d>c_{d-1}>\cdots > c_1$). Indeed, these relations are homogeneous for the weight grading assigning to the variable $c_k$ the weight $k$, the normal forms with respect to the leading terms of the relations are easily seen to be $c_d^l c_k$ with $0\le k\le d$, and for each total weight there is just one normal form of that weight. This shows that the above algebra homomorphism $\tilde{A}\to A$ can only be surjective if there are no other elements in the Gröbner basis, for presence of such elements would inevitably create a dependency between the normal forms we listed. This means that $\tilde{A}\cong A$, and that our presentation of the algebra $\tilde{A}$ is inhomogeneous Koszul \cite{MR1250981}. By a direct inspection, the Koszul dual algebra $\tilde{A}^!$ is precisely the differential graded algebra $\mathfrak{m}_d(S^1)$. Because of the inhomogeneous Koszul property, the algebra $\tilde{A}^!$ is quasi-isomorphic to $\Omega(A^\vee)$ which in turn is quasi-isomorphic to $\sim H_\bullet(S^1)$, and it is clear that the map $\phi_d$ implements an explicit quasi-isomorphism, as required. 
\end{proof}

\begin{remark}
Our argument shows that for $d>1$ it is not possible to get a model for $D_{2d}\rtimes S^1$ by merely taking the homology, since the fundamental class $[S^1]$ would act trivially on $\Gerst_{2d}(2)$, and hence on $\Gerst_{2d}$, by degree reasons: the binary operations of $\Gerst_{2d}$ have degrees $0$ and $2d-1$. Let us give a useful elementary hint for how the higher homotopies of the circle action enter the story. Let us consider the Hopf-type fibration $S^{2d-1}\to \mathbb{C}P^{d-1}$ with the fiber $S^1$. The Serre spectral sequence of this fibration has
 \[
E_2^{p,q}=H^p(\mathbb{C}P^{d-1})\otimes H^{q}(S^1)=\k[\xi,c]/(c^p)    
 \]
with $|\xi|=-1$, $|c|=-2$. It converges to $H^\bullet(S^{2d-1})$, and hence, it must have $d_2(\xi)=c$. The action of the circle corresponds to a coaction of $H^\bullet(S^1)$, and its higher homotopies are obtained by homotopy transfer formulas \cite{LV}, which, once one chooses a contracting homotopy $h$ for the above differential $d_2$, immediately shows that the $d$-th higher operation sends $1$ to the linear dual of the fundamental cycle $[S^{2d-1}]$. We find this elementary homotopy transfer exercise very illuminating.
\end{remark}

\section{The Givental action}\label{sec:Givental}

\subsection{An infinitesimal action of \texorpdfstring{$\Der(\calO)[[z]]$}{DerOz} on maps from \texorpdfstring{$\HyperCom_d$}{HCd} to \texorpdfstring{$\calO$}{O}}

Let $\calO$ be an arbitrary operad. Suppose that $f\in\Hom(\HyperCom_d,\calO)$ is a morphism of operads. Clearly, $f$ makes $\calO$ a $\HyperCom_d$-bimodule, so we may talk about $f$-derivations of $\HyperCom_d$ with values in $\calO$: we say that a map  $D\colon\HyperCom_d\to\calO$ is an \emph{$f$-derivation} if 
 \[
D(u\circ_\star v)=D(u)\circ_\star f(v)+ f(u)\circ_\star D(v)    
 \]
for all $u\in\HyperCom_d(I\sqcup\{\star\})$ and all $v\in \HyperCom_d(J)$. 
We denote by $\Der_f(\HyperCom_d,\calO)$ the vector space of all $f$-derivations of $\HyperCom_d$ with values in $\calO$. Geometrically, that vector space is the tangent space to the morphism $f$ in the following precise sense: we have $D\in \Der_f(\HyperCom_d,\calO)$ if and only if
 \[
f+\epsilon D\in \Hom_{\k[\epsilon]/(\epsilon^2)}(\HyperCom_d[\epsilon]/(\epsilon^2),\calO[\epsilon]/(\epsilon^2)) .
 \]
Here and below we shall view the set $\Hom_{\mathrm{Op}}(\HyperCom_d,\calO)$ of all morphisms of operads from $\HyperCom_d$ to $\calO$ simply as an algebraic variety; though in reality it is infinite-dimensional (we need to define images of infinitely many generators and ensure that they satisfy infinitely many identities, which are polynomial equations on coordinates of the images), and one should be talking about pro-varieties, the rather basic geometric concepts and results that we require in this paper can be adapted quite trivially to that level of generality.

\begin{proposition}\label{prop:GiventalDer}
For any $a\in \Der(\calO)[[z]]$, there is a unique vector field $Y(a)$ on the variety $\Hom_{\mathrm{Op}}(\HyperCom_d,\calO)$ satisfying 
\begin{equation}\label{eq:GiventalDer}
[(Y(Dz^p))(f)](\mu_{n}^{(k)})= 
\sum_{\substack{p_0+\cdots+p_q+q=p\\S_0\sqcup \cdots\sqcup S_{q}=\underline{n}}}
f(\mu^{(k+p_{q})}_{S_{q}\sqcup\{\star\}})\circ_\star f(\mu^{(p_{q-1})}_{S_{q-1}\sqcup\{\star\}})\circ_\star\cdots\circ_\star f(\mu^{(p_1)}_{S_1\sqcup\{\star\}}) \circ_\star D(f(\mu^{(p_0)}_{S_0})).
\end{equation}
The assignment $a\mapsto Y(a)$ is a Lie algebra homomorphism from the Lie algebra $\Der(\calO)[[z]]$ to the Lie algebra of vector fields on the variety $\Hom_{\mathrm{Op}}(\HyperCom_d,\calO)$. In other words, for all $a,a'\in\Der(\calO)[[z]]$, we have
 \[
Y([a,a'])=[Y(a),Y(a')]. 
 \] 
\end{proposition}

\begin{proof}
Let us denote by $\calT_d$ the free operad generated by all operations $\mu_n^{(p)}$ of which the operad $\HyperCom_d$ is a quotient (by the relations described in Theorem \ref{th:CGK}). The map of operads $f\colon\HyperCom_d\to\calO$ induces a map of operads $\calT_d\to\calO$, which we shall denote by the same letter $f$. Since the operad $\calT_d$ is free, any map of its species of generators to $\calO$ extends to a unique $f$-derivation. Thus, to establish that the vector field $Y(Dz^p)$ is well defined, it is enough to show that the $f$-derivation $Dz^p.f\in\Der_f(\calT_d,\calO)$ vanishes on all relations of $\HyperCom_d$. It is enough to show that it vanishes on all generators of the ideal of relations. Let us take one of those generators, say,
 \[
\sum_{q\ge 1}\sum_{\substack{ p_1+\cdots+p_q+q-1=m\\S_1\sqcup \cdots\sqcup S_q=\underline{n},\\S_1,\ldots,S_q\ne\varnothing\colon\{i,j\}\subseteq S_1
}}
\mu^{(p_q)}_{S_q\sqcup\{\star\}}\circ_\star\cdots\circ_\star\mu^{(p_1)}_{S_1}
 \]
for some $n$, some chosen elements $i\ne j\in\underline{n}$, and some $m\ge d$, and apply our $f$-derivation to it, amounting to replacing it by   
 \[
\sum_{q\ge 1}\sum_{\substack{p_1+\cdots+p_q+q-1=m\\S_1\sqcup \cdots\sqcup S_q=\underline{n},\\S_1,\ldots,S_q\ne\varnothing\colon\{i,j\}\subseteq S_1
}}\sum_{u=1}^q 
f(\mu^{(p_q)}_{S_q\sqcup\{\star\}})\circ_\star \cdots \circ_\star(Dz^p.f)(\mu^{(p_u)}_{S_u\sqcup\{\star\}})\circ_\star \cdots \circ_\star f(\mu^{(p_2)}_{S_2\sqcup\{\star\}})\circ_\star f(\mu^{(p_1)}_{S_1})
 \]
and expanding $(Dz^p.f)(\mu^{(p_u)}_{S_u\sqcup\{\star\}})$ according to Formula \eqref{eq:GiventalDer}. We shall now analyze this expansion in detail, and show that it simplifies to zero. To organize the calculations, we split contributions of Formula \eqref{eq:GiventalDer} into two parts: the sum where we only use the $q=0$ term of Formula \eqref{eq:GiventalDer} everywhere, and the sum of all other terms.
The first part is 
 \[
\sum_{q\ge 1}\sum_{\substack{p_1+\cdots+p_q+q-1=m\\S_1\sqcup \cdots\sqcup S_q=\underline{n},\\S_1,\ldots,S_q\ne\varnothing\colon\{i,j\}\subseteq S_1
}}\sum_{u=1}^q 
f(\mu^{(p_q)}_{S_q\sqcup\{\star\}})\circ_\star \cdots \circ_\star D(f(\mu^{(p_u+p)}_{S_u\sqcup\{\star\}}))\circ_\star \cdots \circ_\star f(\mu^{(p_2)}_{S_2\sqcup\{\star\}})\circ_\star f(\mu^{(p_1)}_{S_1}),    
 \]
and if we subtract from it the action of the derivation $D$ on the relation corresponding to the same $n$, same chosen elements $i\ne j\in\underline{n}$, and $m+p$, we obtain  
 \[
-\sum_{q\ge 1}\sum_{\substack{p_1+\cdots+p_q+q-1=m\\S_1\sqcup \cdots\sqcup S_q=\underline{n},\\S_1,\ldots,S_q\ne\varnothing\colon\{i,j\}\subseteq S_1\\ 
1\le u\le q \colon p_u<p 
}}
f(\mu^{(p_q)}_{S_q\sqcup\{\star\}})\circ_\star\cdots\circ_\star D(f(\mu^{(p_u)}_{S_u\sqcup\{\star\}}))\circ_\star\cdots\circ_\star f(\mu^{(p_1)}_{S_1}).
 \]
To analyze the second part
 \[
\sum_{\substack{q\ge 1\\ p_1+\cdots+p_q+q-1=m\\S_1\sqcup \cdots\sqcup S_q=\underline{n},\\S_1,\ldots,S_q\ne\varnothing\colon\{i,j\}\subseteq S_1,\\  u=1,\ldots,q
}} 
f(\mu^{(p_q)}_{S_q\sqcup\{\star\}})\circ_\star \cdots \circ_\star 
\sum_{\substack{t_0+\cdots+t_{q_u}+q_u=p\\T_0\sqcup \cdots\sqcup T_{q_u}\subseteq \underline{n}\\ q_u>0}}
f(\mu^{(p_u+t_{q_u})}_{T_{q_u}\sqcup\{\star\}})\circ_\star \cdots\circ_\star  D(f(\mu^{(t_0)}_{T_0}))
\circ_\star \cdots \circ_\star f(\mu^{(p_1)}_{S_1}),    
 \] 
we shall split it further, according to where the elements $i$ and $j$ appear. Concretely, they may be contained in the same block ($S_i$ or $T_j$), or not, in which case $i$ may appear ``deeper'' than~$j$, or it may happen the other way round. The terms where $i$ appears deeper than $j$ can be organized in groups that vanish: such group is precisely the sum of terms where the operations substituted into the slot $\star$ of the operation that has $j$ as an argument are exactly the same; forgetting these substituted operations, we recover one of the relations of $\HyperCom_d$ with the chosen arguments $j$ and $\star$.  The same applies to the terms where $j$ appears deeper than~$i$. It remains to consider the terms where $i$ and $j$ are in the same block. In this case, it is possible that the operation that has $i$ and $j$ as its arguments has a slot labelled $\star$, in which case the sum of all operations like that where the operations substituted into the slot $\star$ are exactly the same is a consequence of a relations of smaller arity with the chosen arguments $i$ and $j$. Finally, we are left with the sum of terms where $i$ and $j$ are in the same block which is the deepest block that there is. These terms are precisely those appearing with the opposite sign after the simplification of the very first part of the sum above (for that, it is useful to note that since the second part only has the ``decomposable'' terms of Formula \eqref{eq:GiventalDer}, this forces the inequality $p_u<p$ where the derivation $D$ is applied). We conclude that all terms cancel, as claimed. 
The claim about the Lie algebra homomorphism follows by a direct inspection from Formula \eqref{eq:GiventalDer} and the $f$-derivation property.  
\end{proof}

The result we have just proved has the following crucial corollary. 

\begin{corollary}\label{cor:GiventalAdjoint}
There is a unique Lie algebra homomorphism from $\calO(\underline{1})[[z]]$ to the Lie algebra of vector fields on the variety $\Hom_{\mathrm{Op}}(\HyperCom_d,\calO)$ sending $a\in \calO(\underline{1})[[z]]$ to a vector field $X(a)$ satisfying
\begin{equation}\label{eq:GiventalUnary}
[(X(rz^p))](f) (\mu_{n}^{(k)})= 
\sum_{\substack{p_0+\cdots+p_q+q=p\\S_0\sqcup \cdots\sqcup S_{q}\subseteq \underline{n}}}
f(\mu^{(k+p_{q})}_{S_{q}\sqcup\{\star\}})\circ_\star f(\mu^{(p_{q-1})}_{S_{q-1}\sqcup\{\star\}})\circ_\star\cdots\circ_\star f(\mu^{(p_1)}_{S_1\sqcup\{\star\}}) \circ_\star [r,f(\mu^{(p_0)}_{S_0})].
\end{equation}
\end{corollary}

\begin{proof}
This immediately follows from Proposition \ref{prop:GiventalDer} and the fact that for any $r\in\calO(\underline{1})$, the endomorphism $D=[r,-]$ of $\calO$ is a derivation (such derivations are called inner operadic derivations), and this assignment defines a Lie algebra morphism from $\calO(\underline{1})$ to $\Der(\calO)$, so we may put $X(rz^p):=Y([r,-]z^p)$. 
\end{proof}

For a Lie algebra concentrated in degree zero, a morphism into the Lie algebra of vector fields on a manifold may be viewed as an infinitesimal action. One way to ensure that one can integrate it to a group action is to require that the Lie algebra is pronilpotent, in which case the exponential map may be used. In our case, we may consider elements of degree zero, and avoiding constant terms, getting a pronilpotent Lie algebra
 \[
\mathfrak{g}_{\calO}:=\left\{\sum_{i\ge 1} r_iz^i\colon |r_i|=2i\right \}\subset \calO(\underline{1})[[z]].   
 \]
and the corresponding prounipotent group $G_{\calO}=\exp(\mathfrak{g}_{\calO})$. The following result is a version of the celebrated Givental action on cohomological field theories.

\begin{corollary}\label{cor:GiventalAction}
The group $G_{\calO}$ acts on the variety $\Hom_{\mathrm{Op}}(\HyperCom_d,\calO)$.
\end{corollary}

\begin{proof}
Thanks to Corollary \ref{cor:GiventalAdjoint}, we have an infinitesimal action of a pronilpotent Lie algebra $\mathfrak{g}_{\calO}$ on that variety, which we may exponentiate to an action of the corresponding prounipotent group.
\end{proof}

\subsection{The Givental action and the $\psi$-classes} \label{sec:PsiClasses}

In \cite[Sec.~5]{lindström2025loggeometricmodelslittle}, Lindström explains that the space $\bCGK_d(n)$ has $n+1$ line bundles $\calL_0,\dots,\calL_n$ such that the pull-back of the normal bundle of the stratum $D_J$ (obtained by the operadic composition map $\circ_\star\colon \CGK_d(I\sqcup\{\star\})\times \CGK_d(J) \to \CGK_d(n)$) is isomorphic to $p_1^*\calL_\star \otimes p_2^*\calL_0$, where $p_1$ and $p_2$ are the projections to the first and the second factors in $ \bCGK_d(I\sqcup\{\star\})\times \bCGK_d(J)$, see~\cite[Lemma 5.4]{lindström2025loggeometricmodelslittle}.

One can use the Fulton--MacPherson construction applied to $\mathbb{A}^d$ to compute on $\bCGK_{d}(n)$ the first Chern classes of the dual line bundles $\calL_i^*$, $i=0,\dots,n$,  as it is done in~\cite[Sec.~4.1]{lindström2025loggeometricmodelslittle}. Alternatively, we refer the reader to~\cite[Sec.~9]{nesterov2025hilbertschemespointsfultonmacpherson}, where the computation of the first Chern classes of the $d$-dimensional bundles can be adapted to our needs using the diagonal embedding $S^1=U(1)\to U(d)$ of the corresponding structure groups; the restriction of those formulas to the relevant stratum of the Fulton--MacPherson compactification is completely standard. In our notation, we have:
 \[
   \psi_0:= c_1(\calL_0^*) = -x_{\underline{n}}, \qquad
   \psi_i:= c_1(\calL_i^*)  = x_{\underline{n}} + \sum_{i\in  I\subsetneq\underline{n}}x_I,     
 \]
which implies that $\mu_n^{(m)} = (-1)^m \psi_0^m \mu_n^{(0)}$ and 
 \[
(\psi_0 + \psi_i)\mu_{n}^{(m)} = \sum_{\substack{I\sqcup J = \underline{n}\\ i\in J, I\not=\emptyset, |J|\geq 2}}\mu_{I\sqcup \star}^{(m)} \circ_{\star} \mu_J^{(0)}, \quad m\geq 0. 
 \]
These formulas allow us to revisit the formulas for the Givental action from Corollary \ref{cor:GiventalAdjoint}. Namely, by a direct computation one may establish a formula that is more recognizable to the experts in Gromov--Witten theory and is directly aligned with the usual formula in $d=1$ case, see, e.~g., \cite[Sec. 4.3]{MR4580527}. That said, this formula is less general than ~\eqref{eq:GiventalDer} and requires more notation than~\eqref{eq:GiventalUnary}, so we record the end result and omit the explicit computation leading to that result.
\begin{proposition} The vector field $X(a)$ given by~\eqref{eq:GiventalUnary} can equivalently be rewritten, for $a=rz^p$, as 
   \begin{equation*}
      X(rz^p)(f)(\mu_{n}^{(k)})=   r \circ_{1} f(\mu_{n}^{(k+p)}) -\sum_{i=1}^n f(\psi_i^p\mu_{n}^{(k)}) \circ_i r + 
      \sum_{\substack{s+t=p-1\\I\sqcup J=\underline{n}}}
      f(\psi_{\star}^s\mu^{(k)}_{I\sqcup\{\star\}})\circ_\star \big( r \circ_1  f(\mu^{(t)}_{J})\big).
   \end{equation*}
\end{proposition}

\subsection{Operators of fractional differential order}

In the spirit of \cite[Th. 1]{MR3019721} it is natural to discuss the concept of the differential order of operators on a commutative associative algebra that emerges naturally from Proposition \ref{prop:GiventalDer}. For that, it is convenient to discuss everything in terms of homotopy circle actions. We say that $\{\Delta_i\}_{i\ge 0}$ is a homotopy circle action on a graded vector space $V$ if $|\Delta_i| = 2i-1$ and  
 \[
\sum_{i=0}^{n}\Delta_i\Delta_{n-i}=0.    
 \]
If we introduce formal variable $z$ with $|z|=-2$, such a structure may be faitfully encoded by a formal expression $\Delta=\sum_{i\ge 0}\Delta_iz^i$  satisfying $|\Delta|=-1$ and $\Delta^2=0$. 

\begin{definition}\label{def:FractionalOrder} 
Let $(V,\mu)$ be a graded commutative associative algebra and let $\{\Delta_i\}_{i\ge 0}$ be a homotopy circle action on $V$. We say that it is of \emph{homotopical differential order at most $\frac{d+1}{d}$} if each operator $\Delta_i$ is of differential order at most $\lfloor i/d\rfloor+1$. 
\end{definition}

In order words, $\Delta_0,\dots,\Delta_{d-1}$ are of differential order at most~$1$, $\Delta_d,\dots,\Delta_{2d-1}$ are of differential order at most~$2$, etc. This definition is very natural in the context of commutative homotopy BV algebras considered by Kravchenko~\cite{Kra}. Indeed, in this case one considers a specific homotopy transfer of a BV algebra structure that only produces non-trivial homotopies assembling into a homotopy circle action $\Delta$; for that action, the homotopical differential order is at most $2$. Furthermore, applying analogous construction to a $BV_k$-algebra (these are analogs of BV algebras where the BV operator is assumed to have the differential order $\leq k$, see~\cite[Def.~4.1.2]{MR4580527}), one induces a homotopy circle action $\Delta$ of homotopical differential order at most $k$, $k=1,2,3,\dots$.

Similarly, applying to the model of the operad $\D_{2d}\rtimes S^1$ that we use (see Proposition~\ref{prop:d-model}) a homotopy transfer that only produces non-trivial homotopies assembling into a homotopy circle action $\Delta$, we obtain a graded commutative associative algebra with a homotopy circle action of homotopical differential order at most $\frac{d+1}{d}$. Let us prove an analogue of results of~\cite{MR4580527,MR3019721} giving a characterization of fractional homotopical differential order.

\begin{proposition}\label{prop:GiventalDO} 
A homotopy circle action $\{\Delta_i\}_{i\ge 0}$ on a graded commutative associative algebra $(V,\mu)$ is of homotopical differential order at most $\frac{d+1}{d}$ if and only if 
 \[
X(\Delta)(f) = 0
 \]
for the operad morphism $f\colon\HyperCom_d\to \calE{}nd_V$ satisfying $f(\mu_n^{(k)})=0$ unless $k=d(n-1)-1$. 
\end{proposition}

\begin{proof}
Note that the morphism $f$ annihilates all elements of $\HyperCom_d$ of non-zero homological degree, and therefore Formula \ref{eq:GiventalUnary} implies that the contributions of $X(\Delta_iz^i)(f)(\mu_n^{d(n-1)-1})$ for different $i$ have different homological degrees, so $X(\Delta)(f) = 0$ implies $X(\Delta_iz^i)(f) = 0$ for each $i$ individually. Moreover, an easy inductive argument using the defining relations of the operad $\HyperCom_d$ shows that all $n$-ary operations obtained as iterations of $\mu_2^{(d-1)}$ are equal to each other and equal to $(-1)^n\mu_n^{(d(n-1)-1)}$. 

Suppose now that $s\le i/d<s+1$, so that $\lfloor i/d\rfloor+1=s+1$. Using the assumption $f(\mu_n^{(k)})=0$ unless $k=d(n-1)-1$ and the abovementioned formula for the iterations of $\mu_2^{(d-1)}$, we may simplify the expression of Formula \eqref{eq:GiventalUnary} for $X(\Delta_i z^i)(\mu_{s+2}^{d(s+1)-1-i})$. Essentially, all terms that appear there are of the form $\mu_{I\sqcup\{\star\}}^{(d|I|-1)}\circ_\star(\Delta_i\circ_1)\mu_{J}^{(d(|J|-1)-1)}$, with the coefficient equal, up to a sign, to the Euler characteristic of the complex computing the reduced homology of a certain interval in the Boolean lattice. This way, we obtain precisely the Koszul brace $\langle-,\ldots,-\rangle_{s+2}^{\Delta_i}$, proving that the operator $\Delta_i$ is of differential order at most $s+1$, as required.  
\end{proof}

\begin{remark}
In the light of this proposition and Definition~\ref{def:FractionalOrder}, it is natural to propose the name ${\frac{d+1}{d}}\text{-}BV$ for the model of the operad $\D_{2d}\rtimes S^1$ introduced above.
\end{remark}

\section{The Chen--Gibney--Krashen operad as a homotopy quotient}\label{sec:CGKHomotopyQuotient}

The goal of this section is to prove the following result. 

\begin{theorem}\label{th:homotopy-q}
We have a quasi-isomorphism of differential graded operads
 \[
(\HyperCom_d,0)\to {\textstyle\frac{d+1}{d}}\text{-}BV/\Delta,     
 \]
where the latter operad denotes the homotopy quotient of the operad $C_\bullet(\D_{2d}\rtimes S^1)$ by its natural circle action.
\end{theorem}

Various versions of this result for $n=1$ were proved by Drummond-Cole and Vallette \cite{MR3029946}, Khoroshkin and the first author \cite{MR3084563}, Khoroshkin, Markarian and the third author \cite{MR3079329}, two of the authors of the present paper and Vallette \cite{MR3426748}, Drummond-Cole \cite{MR3252959},  Tu \cite{tu2021categoricalenumerativeinvariantsground}, as well as Oancea and Vaintrob \cite{oancea2025delignemumfordoperadtrivializationcircle}. For $d>1$, the situation is a bit more intricate than in the $d=1$ case, due to the fact that the operad $\D_{2d}\rtimes S^1$ is not formal. Our proof will combine convenient features of many different proofs that exist for $d=1$, and we believe that staying within the context of just one of those proofs comes at a price: it either makes the argument more complicated, or does not construct a direct explicit quasi-isomorphism. 

\begin{proof}
Let us begin by noting that, completely analogously to \cite{MR3079329}, one can show that a differential graded operad ${\frac{d+1}{d}}\text{-}BV/\Delta$ modelling the homotopy quotient of ${\frac{d+1}{d}}\text{-}BV$ by the circle action may be obtained by adjoining to ${\frac{d+1}{d}}\text{-}BV$ a trivialization of the circle action, that is, unary operations $\phi_k$ with $|\phi_k|=2k$ for all $k\ge 1$, the differential $\partial$ on which is uniquely determined by the equation
 \[
\exp(-\phi(z))\partial\exp(\phi(z))=\partial+\Delta_1z+\cdots+\Delta_dz^d.    
 \] 
Our argument consists of three parts. First, we shall use the Givental action to construct an explicit morphism of differential graded operads from $(\HyperCom_d,0)$ to ${\frac{d+1}{d}}\text{-}BV/\Delta$. As our second step, we shall show by an algebraic argument that this morphism is injective on the homology. Finally, we shall employ a geometric argument to show that the two differential graded operads are quasi-isomorphic as chain complexes, thus completing the proof. 

There is a uniquely defined map $\theta_0$ from the operad $\HyperCom_d$ to the underlying operad of ${\frac{d+1}{d}}\text{-}BV/\Delta$, which we shall denote $\calO$ for brevity, for which $\theta_0(\mu_n^{(k)})=0$ unless $k=d(n-1)-1$, and $\theta_0(\mu_2^{(d-1)})=-\cdot-$. Thanks to Corollary \ref{cor:GiventalAction}, we may apply the Givental action of the element $\phi(z)\in G_{\calO}=\exp(\mathfrak{g}_{\calO})$ and define a map $\theta$ between the same operads for which 
 \[
\theta=\exp(X(\phi(z)))(\theta_0),    
 \]
where $X(-)$ is the vector field defined in Corollary \ref{cor:GiventalAdjoint}. 
Proposition \ref{prop:GiventalDO} implies that the semi-direct product relations can be written as  
 \[
X(\Delta_1z+\cdots+\Delta_dz^d)(\theta_0)=0.    
 \]   
Moreover, viewing the derivation $\partial$ as $\partial z^0$, we find that $X(\partial z^0)(f)=\partial(f)$, so we may write   
\[
X(\partial z^0+\Delta_1z+\cdots+\Delta_dz^d)(\theta_0)=0.    
 \]
Using the fact that $X(-)$ is a Lie algebra homomorphism, we find 
 \[
X(\partial z^0)(\theta)=X(\partial z^0)(\exp(X(\phi(z)))(\theta_0))=\exp(X(\phi(z)))(X(\partial z^0+\Delta_1z+\cdots+\Delta_dz^d)(\theta_0))=0 ,  
 \]
and therefore the map $\theta$ is compatible with the differentials, so it is a map of differential graded operads. This completes the first step of the proof.

Let us show that $\theta$ is injective on the homology. Examining the proof of Proposition \ref{prop:Koszuldual}, one immediately sees that the quadratic leading terms of relations for the operad $\widetilde\HyperCom_d$ are precisely the elements $m^k_{\underline{n}\setminus\{i,j\}\sqcup\{\star\}}\circ_\star m^{d-1}_{\{i,j\}}$ with $\{i,j\}\ne \{n-1,n\}$ and $k=0,\ldots,d-1$. Examining the proof of Theorem \ref{th:CGK}, we see that the same is true for the minimal set of generators $\mu^{(p)}_n$, $0\le p\le d-1$, of the operad $\HyperCom_d$, if we start our comparison by postulating that tree monomials with fewer internal vertices are larger, and we have the following inductive definition of a basis $\mathsf{B}_d$ of the operad $\HyperCom_d$ analogous to the basis of the operad $\HyperCom$ from \cite[Prop.~3]{MR4392457}:
\begin{itemize}
\item[\textbullet] the unit element belongs to $\mathsf{B}_d$,
\item[\textbullet] a shuffle composition $\mu_k^{(p)}(\tau_1,\ldots,\tau_k)$ belongs to $\mathsf{B}_d$ if and only if $\tau_1$, \ldots, $\tau_k$ belong to $\mathsf{B}_d$, and the top level operations (root vertices) of $\tau_1$, \ldots, $\tau_{k-1}$ are different from $\mu_2^{(d-1)}$. 
\end{itemize}
Since elements that do not contain any $\Delta_i$ are not contained in the image of the differential, to prove injectivity of $\theta$ on the homology it is enough to prove its injectivity on the nose, that is, to show that the images of elements of $\mathsf{B}_d$ are linearly independent in the operad ${\frac{d+1}{d}}\text{-}BV/\Delta$. The suboperad of ${\frac{d+1}{d}}\text{-}BV/\Delta$ generated by the commutative associative product $-\cdot-$ and all the operations $\{\phi_i\}$ has an obvious basis $\overline{\mathsf{B}}_d$ constructed as follows:
\begin{itemize}
\item[\textbullet] the unit element belongs to $\overline{\mathsf{B}}_d$,
\item[\textbullet] every element $\phi_i\circ_1 \tau$ with $\tau\in \overline{\mathsf{B}}_d$ belongs to $\overline{\mathsf{B}}_d$,  
\item[\textbullet] a shuffle composition $\tau_1\cdot \tau_2$ belongs to $\overline{\mathsf{B}}_d$ if and only if $\tau_1$ and $\tau_2$ belong to $\overline{\mathsf{B}}_d$, and the root vertex of $\tau_1$ is not labelled with $-\cdot-$. 
\end{itemize}

If we prioritize trees with fewer internal vertices in the operad ${\frac{d+1}{d}}\text{-}BV/\Delta$ as well, the image of $\mu_n^{(p)}$ under the morphism $\theta$ has the leading term $(-1)^n\phi_{d(n-1)-1-p}\circ_1 \mu_n^{(d(n-1)-1)}$. We note that the number $N=d(n-1)-1-p=d(n-2)+(d-1-p)$ allows us to reconstruct $n$ and $p$ uniquely, since we are considering the minimal system of generators, and hence $0\le p\le d-1$, and thus $n-2=\lfloor N/d\rfloor$ and $d-1-p=N-(n-2)d$. It follows that the leading terms of all basis elements of $\mathsf{B}_d$ are distinct basis elements of $\overline{\mathsf{B}}_d$, completing the proof of injectivity.

Finally, let us show that for each $n\ge 2$, we have an isomorphism of homology of the chain complexes $(C_*(\D_{2d}\rtimes S^1)/_h S^1)(n)$ and $(C_*(\bCGK_d))(n)$ (as indicated above, for this step we may ignore compatibility of isomorphisms with any structure on the homology). We already recalled in Section \ref{sec:PsiClasses} that $\bCGK_d(n)$ has $n+1$ line bundles $\calL_0,\dots,\calL_n$ such that the pull-back of the normal bundle of the stratum $D_J$ under the respective map $\circ_\star$ is isomorphic to $p_1^*\calL_\star \otimes p_2^*\calL_0$, where $p_1$ and $p_2$ are the projections to the first and the second factors in $\bCGK_d(I\sqcup\{\star\})\times \bCGK_d(J)$. We shall use these bundles to mimic the the construction of Kimura--Stasheff--Voronov~\cite{MR1341693} that is available in the case $d=1$. Namely, let $\hCGK_d(n)$ be the real oriented blow-up of $\bCGK_d(n)$ along the boundary $\bCGK_d\setminus \CGK_d$ (or, in other words, along the full image of the operadic composition in $\bCGK_d$) and let $\hCGK^{\mathrm{fr}}_d(n)$ denote the total space of the $(S^1)^{n+1}$-bundle over $\hCGK_d(n)$, where the individual $S^1$ factors are the $S^1$-bundles associated with $\calL_0,\dots,\calL_n$. The collection of spaces $\hCGK^{\mathrm{fr}}_d(n)$ is naturally endowed with a structure of a topological $S^1$-operad inherited from the operad structure on $\bCGK_d$. This construction is performed by Lindström~\cite{lindström2025loggeometricmodelslittle} using analytification of the log geometric extension of $\bCGK_d$, generalizing how it is done for $d=1$ by Vaintrob~\cite{vaintrob2019moduliframedformalcurves,vaintrob2021formalitylittledisksalgebraic}. In particular, according to~\cite[Th.~5.9]{lindström2025loggeometricmodelslittle}, the $S^1$-operad $\hCGK^{\mathrm{fr}}_d$ is homeomorphic to $ \D_{2d}\rtimes S^1 $. We use this realization of $\D_{2d}\rtimes S^1$ to employ the approach of Tu~\cite{tu2021categoricalenumerativeinvariantsground} to the homotopy quotient.

There are two steps in the argument of \cite{tu2021categoricalenumerativeinvariantsground} that shall use. First, from an $S^1$-operad one passes to its Feynman compactification (the terminology reflects the fact that \cite{tu2021categoricalenumerativeinvariantsground} works with modular operads). For that, we consider the chain operad $C_\bullet(\hCGK^{\mathrm{fr}}_d)$ and form a new operad $\calF C_\bullet(\hCGK^{\mathrm{fr}}_d)$ modelled by a direct sum over stable rooted trees whose vertices with $m$ inputs are decorated by homotopy $S^1$-quotients of $C_\bullet(\hCGK^{\mathrm{fr}}_d)(m)$ considered as $(S^1)^{m+1}$-modules, cf.~\cite[Sec.~2.4-2.7]{tu2021categoricalenumerativeinvariantsground}. The operad $\calF C_\bullet(\hCGK^{\mathrm{fr}}_d)$ gives a model for $C_\bullet(\hCGK^{\mathrm{fr}}_d/_h S^1)$. Second, one may use the resolution of Mondello~\cite{Mondello} initially tailored for homotopically inverting the real oriented blow-up of a normal crossing divisor in the realm of simplicial spaces. In our situation, repeating \emph{mutatis mutandis} all steps of ~\cite[Sec.~3.4-3.8]{tu2021categoricalenumerativeinvariantsground}, we obtain for each $n\ge 2$ a chain complex $\mathrm{Tot}(\calF C_\bullet(\hCGK^{\mathrm{fr}}_d(n))_\bullet)$; an analogue of \cite[Th.~3.6]{tu2021categoricalenumerativeinvariantsground} ensures
that it is quasi-isomorphic to $C_\bullet(\bCGK_{d}(n))$, while an analogue of \cite[Prop.~3.8]{tu2021categoricalenumerativeinvariantsground} ensures that it is quasi-isomorphic to $\calF C_\bullet(\hCGK^{\mathrm{fr}}_d)(n)$, which represents $C_\bullet(\hCGK^{\mathrm{fr}}_d/_h S^1)(n)$, and hence $C_\bullet((\D_{2d}\rtimes S^1)/_h S^1)(n)$. This completes the proof. 
\end{proof}

\begin{remark}
Examining the last part of the above proof, one finds that this line of argument alone is not sufficient to directly conclude that the operads $C_*(\D_{2d}\rtimes S^1)/_h S^1$ and $C_*(\bCGK_d)$ are quasi-isomorphic, and the above argument only furnishes an arity-wise quasi-isomorphism of chain complexes. Thus, some extra arguments are needed to ensure that the operads are quasi-isomorphic, and we believe that our argument accomplishes that goal in a rather succinct way.
\end{remark}

\section{The \texorpdfstring{$d\to\infty$}{dinf} (co)limit of the Chen--Gibney--Krashen spaces}\label{sec:DJcolimit}

In this section, we discuss our original motivation and explain how the main results of the paper relate to it. 

\subsection{The operad \texorpdfstring{$1\text{-}\HyperCom$}{1Hycom}}
As we recalled in the introduction, in the article \cite{MR4580527}, two of the authors of the present paper and Tamaroff introduced a family of operads depending on a positive integer parameter $k$, which they denoted $k\text{-}\HyperCom$. For $k\ge 2$, these operads are somewhat similar to the operad $\HyperCom$ of Getzler: they are generated by symmetric operations $\mu_n$ of all possible arities $n\ge 2$ of certain even homological degrees that satisfy quadratic relations analogous to those of Getzler. (However, their dimensions of graded components do not possess the same agreeable properties, and, for instance, one can check that they are not homology operads of smooth complex projective varieties; a geometrical description of those operads is an open problem.) For $k=1$, the corresponding operad is qualitatively different: it is generated by infinitely many commutative binary operations $\nu_2^{(p)}$ of all possible even homological degrees $2p\ge 0$ that satisfy the following generalized associativity relations:
\begin{equation}\label{eq:1HyCom}
\sum_{p+q=m}\nu_2^{(p)}\circ_1\nu_2^{(q)}=\sum_{p+q=m}\nu_2^{(p)}\circ_2\nu_2^{(q)}, \quad m\ge 0 .    
\end{equation}
There is one almost obvious ``topological'' operad whose homology is exactly this algebraic operad. The quotation marks suggest that it is not like most topological operads one knows, and indeed it is not. Namely, let us consider the operad
 \[
\coprod_{x\in\mathbb{C}P^\infty} \Com,    
 \]
the coproduct of copies of the operad $\Com$ indexed by points of $\mathbb{C}P^\infty$, with obvious topology on it. As an operad, it is generated by commutative binary operations $\nu_x$, $x\in\mathbb{C}P^\infty$, satisfying the relations
\begin{equation}\label{eq:1HyComTop}
\nu_x\circ_1\nu_x=\nu_x\circ_2\nu_x  .   
\end{equation} 
It easily follows that the homology of this operad is generated by commutative binary operations $H_\bullet(\mathbb{C}P^\infty)$. Furthermore, Relation \eqref{eq:1HyComTop} uses the diagonal $\Delta\colon \mathbb{C}P^\infty\to \mathbb{C}P^\infty\times \mathbb{C}P^\infty$, and so it is almost obvious that it induces precisely Relation \eqref{eq:1HyCom} on the homology. However, since an operator of differential order at most $1$ is also an operator of differential order at most $2$, there is a map from the operad $\HyperCom$ to the operad $1\text{-}\HyperCom$; thus, it is desirable to be able to model this map on the topological level. For the coproduct of copies of $\Com$, there does not seem to be any map of this sort. 

\subsection{The Stanley--Reisner rings appear}
In order to make an educated guess for the unknown topological operad behind the operad $1\text{-}\HyperCom$, one can compute the dimensions of graded components of $1\text{-}\HyperCom(n)$ (which would be the Betti numbers of the $n$-th component of the corresponding topological operad). The answer turns out to be given by the elegant formula 
 \[
\frac{P_n(q^2)}{(1-q^2)^n},     
 \] 
where the polynomial $P_n(q^2)$ is the generating function for the elements of the $n$-th row of the second-order Eulerian triangle \cite{oeis-euler}. Moreover, the series $\frac{P_n(q)}{(1-q)^n}$ is known \cite{MR2195566} to be the generating series for dimensions of components of the commutative ring
 \[
\mathbb{Q}[x_I\colon I\subseteq\underline{n}, |I|\ge 2]/(x_Ix_J\colon I\cap J\ne\varnothing, I\not\subseteq J, J\not\subseteq I),     
 \]
which we recognize as the Stanley--Reisner ring of the nested set complex for the minimal building set of the lattice of partitions of $\underline{n}$. 

\begin{definition}
Let $\Gamma$ be a simplicial complex on a vertex set $V$. The \emph{Davis--Januszkiewicz space} $DJ(\Gamma)$ associated to $\Gamma$ is the subset of the product $\prod_{v\in V}\mathbb{C}P^\infty$ obtained as the union of subsets 
 \[
DJ(\sigma)=\prod_{v\in \sigma}\mathbb{C}P^\infty\times\prod_{v\notin\sigma} \{\mathrm{pt}\}    
 \]
for all faces $\sigma$ of $\Gamma$. 
\end{definition}

This definition is a particular case of a construction introduced by Anick \cite[p.~344]{MR764587} under the name $\Gamma$-wedge. The name and the notation, nowadays usual in the literature, refers to the work of Davis and Januszkiewicz~\cite{MR1104531} who in particular proved that the Stanley--Reisner ring $SR(\Gamma)$ associated to $\Gamma$ is the cohomology ring of $DJ(\Gamma)$. 

Thus, we have the operad $1\text{-}\HyperCom$ which seems to be related to the Davis--Januszkiewicz spaces for the nested set complexes associated to the minimal building sets of the partition lattices, and the operad $\HyperCom$, which is the homology operad of the Deligne--Mumford operad; whose components $\overline{\calM}_{0,n+1}$ are the De Concini--Procesi wonderful compactifications of the braid arrangements associated to the very same minimal building sets of the partition lattice. This was a useful starting hint for relating the two pictures, which led us to the following result. 

\begin{theorem}
For $d<d'$, we have a map of operads $\bCGK_d\to \bCGK_{d'}$. For the corresponding colimit, we have 
 \[
H^\bullet(\colim_{d\to\infty}\bCGK_d(n))\cong SR(N(\calL(\Delta_n),\calG_{\min}(\Delta_n)))    
 \]
and  
 \[
H_\bullet(\colim_{d\to\infty}\bCGK_d)=\colim_{d\to\infty}\HyperCom_d\cong 1\text{-}\HyperCom.    
 \]  
\end{theorem}

\begin{proof}
Let us first observe that, because of realizations of the Chen--Gibney--Krashen spaces as wonderful compactifications, we have the maps 
 \[
\bCGK_d(n)(\mathbb{C})\to \overline{Y}_{\calL(\Delta_n^{(d)}),\calG_{n,\min}^{(d)}} \to     
\overline{Y}_{\calL(\Delta_n^{(d')}),\calG_{n,\min}^{(d')}}\cong \bCGK_{d'}(n)     
 \]
for all $d<d'$; since the operad map comes from inclusion of boundary divisors of the wonderful compactification, it is clear that this map is a map of operads. 

From the description of the generators $\sigma_G$ in Section \ref{sec:dcp}, it easily follows that the induced map of cohomology rings 
 \[
H^\bullet(\bCGK_{d'}(n))\to  H^\bullet(\bCGK_{d}(n))   
 \]
sends every generator $\sigma_G$ on the left to the same generator on the right. It then follows from Milnors $\lim^1=0$ theorem \cite[Th.~7.75]{MR385836} that for the space $\colim_d\bCGK_{d}(n)$ the cohomology ring is precisely the Stanley--Reisner ring. (In the next section, we shall prove a stronger assertion that this colimit is homotopy equivalent to the corresponding Davis--Januszkiewicz space.) As we saw in the proof of Theorem \ref{th:CGK}, the minimal set of generators of the operad $\HyperCom_d$ are the elements $\mu_n^{(p)}$ with $0\le p\le d-1$ of degrees 
 \[
|\mu_n^{(p)}|=2(d(n-1)-1-p)=2(d(n-2)+(d-1-p)).    
 \]
The latter formula for the degree shows that only the generators $\mu_2^{(p)}$ ``survive'' in the limit $d\to\infty$; moreover, it is easy to see that the relations satisfied by these generators are precisely the relations of the operad $1\text{-}\HyperCom$, if we let $\nu_2^{(p)}:=\mu_2^{(d-1-p)}$ and let $d\to\infty$. 
\end{proof}

It is worth noting that the operad we thus constructed does not have an obvious relationship to the coproduct of copies of $\Com$ we discussed above; for instance, it is easy to show that there is no map of operads from that coproduct to $\colim_d\bCGK_{d}$ that is a homotopy equivalence on the level of spaces, and, in fact, no map of operads from that coproduct to $\colim_d\bCGK_{d}$ at all.

\begin{remark}
For the diagram we displayed in the introduction, only the top row is sufficiently nontrivial on the level of spaces. Namely, the homotopy quotient of the $S^1$-framing of the operad $\colim_{d\to\infty}\D_{2d}$ by the circle action is given by $\colim\limits_{d\to\infty}\bCGK_d$for each $n$, but the colimit of $\D_{2d}(n)$ is contractible, and the colimit of $\CGK_d(n)$ is the quotient of the latter contractible space by a free circle action, and thus is nothing but $BS^1$. However, if we pass to the corresponding algebraic operads, the bottom left corner of the diagram becomes nontrivial, for the operad in that corner is $\colim_{d\to\infty}\widetilde\Grav_d$, which we define as the $2d$-fold suspension of the operad of Westerland, leading to the property $\colim_{d\to\infty}\widetilde\Grav_d\cong \mathcal{S}1\text{-}\HyperCom^!$.  
\end{remark}

\subsection{Colimits of multiplication by \texorpdfstring{$d$}{PDFstring} and the Davis--Januszkiewicz spaces}

In this section, we prove (a generalization of) the claim made earlier about the relationship between the Chen--Gibney--Krashen spaces and the Davis--Januszkiewicz spaces. Note that the colimit
 \[
Y^{\infty}_{\calL,\calG}:=\colim_d Y_{\calL^{(d)},\calG^{(d)}}  
 \]
is a subset of 
 \[
\colim_{q}\prod_{G\in\calG} \mathbb{P}(V/G^{\perp})^{\oplus q}\cong 
\prod_{G\in\calG}\colim_{q} \mathbb{P}(V/G^{\perp})^{\oplus q} ,   
 \]
we shall use the notation $\mathbb{C}P^{\infty}_{G}$ for the factor 
 \[
\colim_{q} \mathbb{P}(V/G^{\perp})^{\oplus q}
 \] 
corresponding to $G\in\calG$. In this section, we shall show that for any building set $\calG$ of any subspace arrangement $\calL$, the space $Y^{\infty}_{\calL,\calG}$ has the homotopy type of the Davis--Januszkiewicz space 
 \[
DJ_{\calL,\calG}:=\bigcup_{S\in N(\calL,\calG)}\prod_{I\in S}\mathbb{C}P^\infty\times\prod_{I\notin S}\{\mathrm{pt}\}\subset \prod_{G\in\calG} \mathbb{C}P^\infty,  
 \]
where $\mathrm{pt}$ denotes a base point in $\mathbb{C}P^\infty$. The argument is organized as follows. It is clear that we may choose different spaces that are homotopy equivalent to $\mathbb{C}P^\infty$ as different factors without changing the homotopy type of the result. The natural choice is to take the factor corresponding to $G\in\calG$ to be $\mathbb{C}P^{\infty}_{G}$, so that the spaces $DJ_{\calL,\calG}$ and $Y^{\infty}_{\calL,\calG}$ are subspaces of the same space
 \[
\mathbb{P}_{\calL,\calG}:=\prod_{G\in\calG} \mathbb{C}P^\infty_G.     
 \]
The goal of this section is to prove the following result. 

\begin{theorem}\label{thm:spheq}
The subspaces $DJ_{\calL,\calG}$ and $Y^{\infty}_{\calL,\calG}$ of the space~$\mathbb{P}_{\calL,\calG}$ are homotopy equivalent.
\end{theorem}

\begin{proof}
We argue by induction on $|\calG|$. If $\calG$ is empty, we clearly have $Y^\infty_{\calL,\calG}\cong \mathbb{C}P^\infty\cong DJ_{\calL,\calG}$. Suppose that we wish to prove our assertion for $\calG$, and it is proved for all building sets of smaller cardinality. 
We shall now make a choice of a specific element $G\in\calG$. If all elements of $\calG$ are indecomposable (do not admit nontrivial direct sum decompositions with all summands in $\calG$), let us take for $G$ a minimal element of $\calG$, otherwise let us take for $G$ a minimal decomposable element of $\calG$. In each of these cases $\calG':=\calG\setminus\{G\}$ is known to be a building set of a suitable subspace arrangement $\calL'$, see \cite[Prop. 2.3(4)]{MR1366622} for the first case and \cite[Prop. 2.5(1)]{MR1366622} for the second one. Moreover, and $Y_{\calL,\calG}$ is obtained from $Y_{\calL',\calG'}$ by blowing up a subvariety isomorphic to $Y_{\calL_G,\calG_G}$ see \cite[Th. 3.2(4)]{MR1366622} for the first case, and \cite[Th. 3.2(3), Th. 4.3]{MR1366622} for the second one. 

We shall need a general result about colimits of blowups. Let $Z_{0}\subset Z_{1}\subset\cdots$ be a sub-filtration of a filtration $X_{0}\subset X_{1}\subset\cdots$, where all arrows are closed embeddings of smooth complex manifolds. We shall denote $B_{q}:=\mathrm{Bl}_{Z_{q}} X_{q}$ and consider the colimits 
 \[
Z:=\colim Z_{q},\quad   X:=\colim X_{q},\quad  B:=\colim B_{q}.    
 \]
Let $\nu_{q}:=\nu(Z_{q}\to X_{q})$ be the normal bundle of each of the embeddings, and denote $\nu:=\colim \nu_{q}$. We shall use the open (respectively, closed) disk bundle $D^{\circ}(\nu)$ (respectively, $D(\nu)$) associated with $\nu$, as well as the corresponding sphere bundle $S(\nu)$. 

\begin{proposition}\label{prop:limblow}
Suppose that $\mathrm{rk}\nu= \infty$.
Then $B$ is homotopy equivalent to the pushout $\tilde{B}$ of the diagram
\begin{equation}\label{eq:binfdesc}
\begin{tikzcd}
Z\times \mathbb{C}P^{\infty} & Z \arrow{l}\arrow{r} & X ,
\end{tikzcd}
\end{equation}
where the left arrow is defined using a basepoint of~$\mathbb{C}P^{\infty}$. Moreover, a homotopy equivalence is implemented by a map 
 \[
h\colon \tilde{B}\to B    
 \]
whose restrictions of $h$ to $Z\times \mathbb{C}P^{\infty}$ and $X\setminus Z$ are the exceptional divisor embedding and the identity map, respectively.
\end{proposition}
\begin{proof}
Following \cite{MR2163611}, we identify $\mathrm{Bl}_{Z_{q}} X_{q}$ with the pushout of the diagram
 \[
\begin{tikzcd}
\mathbb{P}(\nu_{q})& S(\nu_{q}) \arrow{l}\arrow{r} & X_{q}\setminus D^{\circ}(\nu_{q}) .
\end{tikzcd}   
 \]
Let us consider the commutative diagram
\begin{equation}\label{eq:bigbl}
\begin{tikzcd}
\mathbb{P}(\nu_{q})\arrow[equal]{d} & S(\nu_{q}) \arrow{l}\arrow[equal]{d}\arrow{r} & X_{q}\setminus D^{\circ}(\nu_{q})\arrow{d}\\
\mathbb{P}(\nu_{q}) & S(\nu_{q}) \arrow{l}\arrow{r} & X_{q}
\end{tikzcd}
\end{equation}
In \eqref{eq:bigbl}, the left horizontal arrows are equal to the quotient map by the natural fiberwise $S^1$-action, and the right horizontal arrows are the natural embeddings. The colimit of \eqref{eq:bigbl} with respect to $q$ yields another commutative diagram
\begin{equation}\label{eq:infbl}
\begin{tikzcd}
\mathbb{P}(\nu)\arrow[equal]{d} & S(\nu) \arrow{l}\arrow[equal]{d}\arrow{r} & X\setminus D^{\circ}(\nu)\arrow{d}\\
\mathbb{P}(\nu) & S(\nu) \arrow{l}\arrow{r} & X
\end{tikzcd}
\end{equation}
Since colimits commute, the space $B$ is the colimit of the top row in \eqref{eq:infbl}. Let us show that the right vertical arrow in \eqref{eq:infbl} is a homotopy equivalence. For that, we consider a commutative diagram
\begin{equation}\label{eq:mpush}
\begin{tikzcd}
S(\nu)\arrow{d} & S(\nu) \arrow[equal]{l}\arrow[equal]{d}\arrow{r} & X\setminus D^{\circ}(\nu) \arrow[equal]{d}\\
D(\nu) & S(\nu) \arrow{l}\arrow{r} & X\setminus D^{\circ}(\nu)
\end{tikzcd}
\end{equation}
whose row-wise pushouts are $X\setminus D^{\circ}(\nu)$ and $X$, respectively, and whose induced maps of row-wise pushouts is precisely the right vertical arrow in \eqref{eq:infbl}. Since every row arrow in \eqref{eq:mpush} is a closed embedding, the rows of \eqref{eq:mpush} are cofibrant diagrams.
This implies that the row-wise pushouts equal the homotopy pushouts. By \cite{MR179792}, $\nu$ is a trivial complex vector bundle on $Z$. Since $\nu$ is of infinite rank, the fiber of $S(\nu)$ is a contractible sphere $S^{\infty}$. Therefore, the left vertical arrow of \eqref{eq:mpush} is a homotopy equivalence. By the standard property of homotopy colimits this implies that the induced map of homotopy pushouts is a homotopy equivalence.

Continuing similarly, we note the right horizontal arrows in \eqref{eq:bigbl} and \eqref{eq:infbl} are closed embeddings, the row diagrams in each of these are cofibrant, and hence the homotopy pushouts of the row diagrams coincide with usual pushouts. Therefore, the row-wise pushouts in \eqref{eq:infbl} are homotopy equivalent, and it remains to find the pushout for the bottom row of \eqref{eq:infbl}.
Consider finally the commutative diagram
\begin{equation}\label{eq:trivb}
\begin{tikzcd}
Z\times \mathbb{C}P^{\infty}\arrow{d} & Z \arrow{l}\arrow{d}\arrow{r} & X\arrow[equal]{d}\\
\mathbb{P}(\nu) & S(\nu) \arrow{l}\arrow{r} & X.
\end{tikzcd}
\end{equation}
Here, the left vertical arrow is an isomorphism of fiber bundles and the middle vertical arrow is given by a choice of a section $s$ for $S(\nu)$. (Here we use once again the fact that $\nu$ is trival \cite{MR179792}.)
In \eqref{eq:trivb}, the rows are cofibrant, and the vertical arrows are homotopy equivalences.
Therefore, \eqref{eq:trivb} identifies the pushout in question with the pushout \eqref{eq:binfdesc}.
\end{proof}

Applying this result to the varieties $Z_d:=Y_{\calL_G^{(d)},\calG_G^{(d)}}$ and $X_d:=Y_{\calL^{(d)},\calG^{(d)}}$, we immediately obtain a homotopy equivalence 
 \[
h=h_{\calL,\calG}\colon \tilde{Y}^{\infty}_{\calL,\calG}\to Y^{\infty}_{\calL,\calG},   
 \]
where $\tilde{Y}^{\infty}_{\calL,\calG}$ is the pushout of the diagram
\begin{equation}\label{eq:Ypushout}
\begin{tikzcd}
Y^{\infty}_{\calL_G,\calG_G}\times \mathbb{C}P^{\infty}_{G} & Y^{\infty}_{\calL_G,\calG_G}\arrow{l}\arrow{r} & Y^{\infty}_{\calL',\calG'},
\end{tikzcd}
\end{equation}
in which the left arrow is obtained by putting the basepoint in the missing coordinate, and the right arrow is the colimit of the maps of wonderful compactifications induced by the map $r\colon G^\perp\to V$. 

Let us show that for the Davis--Janiszkiewicz spaces the same recurrence relation is true strictly, not up to homotopy, that is, that 
the Davis--Janiszkiewicz space $DJ_{\calL,\calG}$ coincides with the pushout of the diagram
\begin{equation}\label{eq:binfdesc}
\begin{tikzcd}
DJ_{\calL_G,\calG_G}\times \mathbb{C}P^{\infty}_{G} & DJ_{\calL_G,\calG_G} \arrow{l}\arrow{r} & DJ_{\calL',\calG'} ,
\end{tikzcd}
\end{equation}
where the row maps are obtained by putting the basepoints in the missing coordinates.

Indeed, among nested sets $S\in N(\calL,\calG)$, we may first consider nested sets that do not contain $G$, in which case $S\in N(\calL',\calG')$, and so the corresponding points of $DJ_{\calL,\calG}$ come from $DJ_{\calL',\calG'}$. Otherwise, $G\in S$, in which case for each $H\in S$, the element $H+G$ belongs to $\calG_G$, and 
 \[
G^\perp/(H+G)^{\perp}\cong (G^\perp+H^\perp)/H^\perp\cong (G\cap H)^\perp/H^\perp=V/H^\perp  
 \]
because of minimality of $G$, and therefore for each $H$, the factor 
$\mathbb{C}P^{\infty}_{H}:=\colim_{q} \mathbb{P}(V/H^{\perp})^{\oplus q}$ in $DJ_{\calL,\calG}$ is exactly the same as the factor 
$\mathbb{C}P^{\infty}_{H+G}:=\colim_{q} \mathbb{P}(G^\perp/(H+G)^{\perp})^{\oplus q}$ in $DJ_{\calL_G,\calG_G}$, proving the claim.

To conclude the proof, we note that the following diagram is commutative:
 \[
\begin{tikzcd}[column sep=0.5cm]
Y^{\infty}_{\calL_G,\calG_G}\arrow{r}{}\arrow{rrd}\arrow{dd} & Y^{\infty}_{\calL',\calG'}\arrow{rrd}{}\arrow{dd} & \\
& & Y^{\infty}_{\calL_G,\calG_G}\times \mathbb{C}P^{\infty}_{G}\arrow[swap]{r}{}\arrow{dd} & \tilde{Y}^{\infty}_{\calL,\calG}\arrow{dd} \arrow{dd}\\
\mathbb{P}_{\calL_G,\calG_G} \arrow{r}\arrow{rrd} & \mathbb{P}_{\calL',\calG'} \arrow{rrd} & \\
& & \mathbb{P}_{\calL_G,\calG_G}\times\mathbb{C}P^{\infty}_{G}\arrow{r} & \mathbb{P}_{\calL,\calG}
\end{tikzcd}    
 \]
Indeed, the top square is defined as a pushout diagram, the back and left side squares are commutative, because the top arrows are restrictions of the bottom arrows, and the bottom square is tautologically commutative. The least obvious claim concerns the front square, where the right vertical map is a composition of the homotopy equivalence $h$ and the inclusion of the wonderful compactification. In this case, the commutativity follows from the property of the homotopy equivalence $h$ indicated in Proposition \ref{prop:limblow}. 

The commutativity of this diagram allows us to glue a homotopy equivalence $DJ_{\calL,\calG}\sim Y^{\infty}_{\calL,\calG}$ from the available homotopy equivalences $DJ_{\calL_G,\calG_G}\sim Y^{\infty}_{\calL_G,\calG_G}$ and $DJ_{\calL',\calG'}\sim Y^{\infty}_{\calL',\calG'}$, thus completing the proof.
\end{proof}

\printbibliography
\end{document}